\documentclass[11pt,oneside]{amsart}
\usepackage{amssymb,amsmath}
\usepackage[margin=1.25in]{geometry}

\usepackage{hyperref} 

\theoremstyle{plain}
\newtheorem{thm}{Theorem}
\newtheorem{prop}[thm]{Proposition}
\newtheorem{lem}[thm]{Lemma}
\newtheorem{cor}[thm]{Corollary}
\newtheorem{rem}[thm]{Remark}
\newtheorem{eg}[thm]{Example}
\newtheorem{defn}[thm]{Definition}
\newtheorem{conj}[thm]{Conjecture}

\usepackage{lipsum}
\usepackage{tikz,mathtools,enumitem,comment,marginnote}
\usetikzlibrary{arrows.meta,positioning}

\newcommand{\nU}{U_{\text{not shorted}}}
\newcommand{\cU}{U_{\text{connected}}}
\newcommand{\bA}{\mathbb{A}}

\newcommand{\El}{\operatorname{Elec}}

\newcommand{\cO}{\mathcal{O}}

\newcommand{\cI}{\mathcal{I}}

\newcommand{\Pio}{\mathring{\Pi}}

\newcommand{\Xo}{\mathring{\chi}}
\newcommand{\Xn}{\chi_n}
\newcommand{\wt}{\text{wt}}
\newcommand{\ov}[1]{\overline{#1}}
\newcommand{\Gr}{{\rm Gr}}

\newcommand{\C}{\mathbb{C}}
\newcommand{\image}{\text{Im}}
\newcommand{\spann}{\text{span}}

\renewcommand{\AA}{\mathbb{A}}
\newcommand{\RR}{\mathbb{R}}
\newcommand{\CC}{\mathbb{C}}
\newcommand{\DD}{\mathbb{D}}
\newcommand{\FF}{\mathbb{F}}

\newcommand{\IG}{\textrm{IG}}
\newcommand{\LG}{\textrm{LG}}

\newcommand{\Mia}[1]{\marginpar{\raggedright\scriptsize\color{orange}Mia: #1}}

\title{Algebraic Geometry of Electroid Varieties}

\author{
  Dawei Shen\address{Department of Mathematics, University of Michigan}\email{dwshen@umich.edu} \and 
  Mia Smith\address{Department of Mathematics, University of Michigan}\email{smithmia@umich.edu} \and 
  David E Speyer \address{Department of Mathematics, University of Michigan}\email{speyer@umich.edu}
}


\begin{document}

\begin{abstract}
Recent work of Lam, Bychkov–Gorbounov–Kazakov–Talalaev, and Chepuri--George--Speyer gave a stratification of the totally nonnegative Lagrangian Grassmannian into electroid cells parameterized by cactus networks, paralleling Postnikov's stratification of the totally nonnegative Grassmannian by positroid cells. Electroid varieties arise as an algebro-geometric extension of electroid cells. The combinatorics of these varieties was studied by Lam in 2018. We build on this work and study the geometric properties of electroid varieties. In analogy to results of Knutson, Lam, and Speyer on positroid varieties, we show that electroid varieties are reduced, irreducible, regular in codimension one, compatibly Frobenius split, and form a stratification. We also show a decomposition of certain electroid varieties as a product of two electroid varieties. As a consequence, the grove measurement map that embeds electroid cells can be extended algebraically to embed an algebraic torus.
\end{abstract}

\maketitle

\section{Introduction}

The study of \textit{stratifications} of the Grassmannian dates back to the work of Schubert in the 19th century.  By stratification of a variety, we mean a partition of the variety into a disjoint union of irreducible locally closed subvarieties, which we also call open strata, such that the Zariski closure of an open stratum is the union of open strata. Schubert's work laid the foundation for the \textit{Bruhat stratification} of the Grassmannian, first introduced by Bruhat for classical groups and later generalized by Chevalley. The Bruhat stratification inspired numerous finer stratifications: the \textit{Richardson stratification} \cite{KL,R} and, most recently, \textit{the positroid stratification} \cite{KLS}. In contrast, the \textit{matroid decomposition of GGMS} \cite{GGMS} does not satisfy the properties of a stratification in the technical sense; namely, the closure of an open stratum is not always a union of open strata.

One can also consider the real points of the Grassmannian as a smooth real manifold, endowed with the analytic topology. Inside the real Grassmannian manifold, we define its \textit{totally nonnegative part} as the semi-algebraic subset where all Pl\"ucker coordinates are nonnegative (or equivalently, have the same sign). In this context, one can also consider \textit{stratifications} using a similar definition as above, with Zariski topology replaced by analytic topology. In 2006, Postinikov stratified the totally nonnegative Grassmannian into \textit{positroid cells} parameterized by plabic graphs \cite{P}. Later in 2011, Knutson, Lam, and Speyer \cite{KLS} stratified the Grassmannian variety into \textit{open positroid varieties}. They call the Zariski closure of open positroid varieties \textit{closed positroid varieties} or just \textit{positroid varieties}. This is an extension of Postnikov's stratification of the totally nonnegative Grassmannian in the sense that restricting the former to its totally nonnegative points recovers the latter.

The positroid stratification has many elegant properties.
A closed positroid variety is the Zariski closure of its positroid cell. (Open) positroid varieties are also projections of (open) Richardson varieties and intersections of cyclically permuted (open) Schubert varieties \cite{KLS}. Moreover, they have particularly nice algebro-geometric properties; they are reduced, irreducible, normal, Cohen-Macaulay, and compatibly Frobenius split \cite{KLS}. 

It is here that the story of positroid varieties intersects with another: the story of planar electrical networks developed by Curtis--Ingerman--Morrow \cite{CIM} and Colin de Verdi\`{e}re--Gitler--Vertigan \cite{CGV}. In 2016, Lam compactified the space of electrical networks, and in-so-doing, formalized its conjectured parallels to the totally nonnegative Grassmannian \cite{Lam2018}. Independent work by Chepuri--George--Speyer and Bychkov--Gorbounov--Kazakov--Talalaev later showed that the compactified space is an isotropic Grassmannian abstractly isomorphic to the Lagrangian Grassmannian \cite{BGKT,CGS2021}. Together with work of Lam, this yields a stratification of the totally nonnegative part of the real points of the Lagrangian Grassmannian into electroid cells parameterized by cactus networks. A natural hope is that this stratification of the totally nonnegative part can be extended to a well-behaved stratification of the Lagrangian Grassmannian variety, parallel to the positroid story \cite{KLS}. 

In \cite{Lam2018}, Lam defined (open) electroid varieties as the intersection between (open) positroid varieties and the Lagrangian Grassmannian.
Each electroid cell is the totally nonnegative real points of a corresponding open electroid variety, and the
open electroid varieties set-theoretically partition the Lagrangian Grassmannian.
However, a fundamental understanding of the algebraic geometry of electroid varieties is still lacking. For example, it was unknown if electroid varieties are Zariski closures of electroid cells or if they form a stratification.

 We build on this rich story and study the combinatorial algebraic geometry of electroid varieties. In analogy to the work of Knutson, Lam, and Speyer in \cite{KLS} on positroid varieties, we prove that electroid varieties are reduced, irreducible, regular in codimension one, compatibly Frobenius split, of the expected dimension, and form a stratification extending the stratification by electroid cells. We construct a splicing isomorphism analogous to \cite{gorsky2025,gorsky2024}. We show that the measurement map which embeds electroid cells to totally nonnegative points in the corresponding open electroid variety can be extended algebraically to embed an algebraic torus. We comment that our approach is fundamentally different from the twist map method by Muller and Speyer  \cite{Muller_2017}. The twist map for the maximal open electroid variety is studied in \cite{george2023}, but the extension to all electroid varieties is unclear. We formally state our results at the end of Section $2$.

\subsection*{Acknowledgements.} D. Shen and M. Smith were partially supported by NSF grant DMS-2348799. D. Speyer was partially supported by NSF grant DMS-2246570. The authors are grateful to Thomas Lam for suggesting this project and meaningful guidance throughout. The authors also wish to thank Terrence George for helpful conversations in early stages of the project.

\section{Background}
\subsection{Electrical networks and cactus networks}

Consider a planar graph $\Gamma$ embedded in a disc $\DD$, with $n$ vertices $v_1$, $v_2$, \ldots, $v_n$ arranged clockwise on the boundary. We think of each edge as an electrical conductor with some fixed resistance in $\RR_{>0}$. Imagine putting voltage $V_i$ at boundary vertex $v_i$ and measuring the resulting current $J_i$ flowing out of vertex $v_i$. The map from $(V_1, V_2, \dots, V_n)$ to $(J_1, J_2, \dots ,J_n)$ is given by an $n \times n$ matrix $L$ called the \emph{response matrix}. The response matrix is symmetric, with rows and columns summing to $0$. Two electrical networks are said to be {\textit{electrically equivalent}} if their response matrices coincide.

\begin{figure}
\centering
    \begin{tikzpicture}
        \draw (-3,0) circle (1);
        \draw (3,0) circle (1);
        \draw (-4,0) node [left] {$v_1$};
        \draw (-2,0) node [right] {$v_2$};
        \draw (2,0) node [left] {$v_1$};
        \draw (4,0) node [right] {$v_2$};
        \fill (-4,0) circle (2pt);
        \draw [thick] (-4,0) -- (-2,0) node[pos=.5, above]{$c$};
        \fill (-2,0) circle (2pt);
        \fill (2,0) circle (2pt);
        \fill (4,0) circle (2pt);
    \end{tikzpicture}
    \caption{The electrical networks with response matrix $L=\begin{bsmallmatrix} -c&c \\ 
c&-c \end{bsmallmatrix}$ for $c>0$ and $c=0$.}
    \label{f:networks with n=2}
\end{figure}
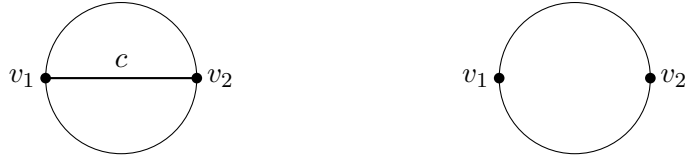

\begin{eg} \label{nIs2}
For $n=2$, all response matrices of electrical networks are of the form
$L=\begin{bsmallmatrix} -c&c \\ 
c&-c \end{bsmallmatrix}$. When $c>0$, $L$ can be realized by the electrical network with a single resistor from $v_1$ to $v_2$ of resistance $c^{-1}$. When $c=0$, $L$ can be realized by the electrical network with two isolated vertices. See Figure \ref{f:networks with n=2}.
\end{eg}

From Example~\ref{nIs2}, we see that the space of possible $2 \times 2$ response matrices (or electrical networks up to equivalence) is isomorphic to $\mathbb{R}_{\geq0}$. Lam~\cite{Lam2018} found a natural way to compactify the space of response matrices, by introducing \emph{cactus networks}. Intuitively, a \emph{cactus network} is an electrical network in which we allow subsets of boundary vertices to be ``shorted" by wires of infinite conductance. Figure \ref{f:cactus} gives two examples of cactus networks. See \cite{Lam2018} for a detailed explanation. Thus, the space of response matrices for $n=2$ cactus networks is $\mathbb{R}_{\geq0}\bigcup \{\infty\}$, homeomorphic to a compact interval. Let $E_n$ denote the space of electrical equivalence classes of (or equivalently, response matrices of) cactus networks with $n$ boundary vertices.
    
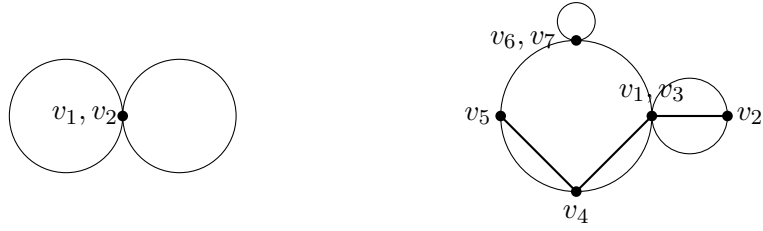
\begin{figure}
    \centering
    \begin{tikzpicture}
        \draw (-2.25,0) circle (.75);
        \draw (-3.75,0) circle (.75);
        \fill (-3,0) circle (2pt);
        
        \draw (3,0) circle (1);
        \draw (4.5,0) circle (.5);
        \draw (3,1.25) circle (.25);
        \fill (4,0) circle (2pt);
        \fill (5,0) circle (2pt);
        \fill (2,0) circle (2pt);
        \fill (3,1) circle (2pt);
        \fill (3,-1) circle (2pt);
        \draw [thick] (4,0) -- (5,0);
        \draw [thick] (4,0) -- (3,-1);
        \draw [thick] (3,-1) -- (2,0);

        \draw (4,0) node [yshift = .3cm] {$v_1,v_3$};
        \draw (5,0) node [xshift = .3cm] {$v_2$};
        \draw (3,-1) node [yshift = -.3cm] {$v_4$};
        \draw (2,0) node [xshift = -.3cm] {$v_5$};
        \draw (3,1) node [xshift = -.7cm] {$v_6,v_7$};
        \draw (-3,0) node [xshift = -.5cm] {$v_1,v_2$};
        
    \end{tikzpicture}
    \caption{The cactus network on the left has boundary vertices $v_1$ and $v_2$ shorted and realizes the limit of the response matrix $L=\begin{bsmallmatrix} -c&c \\ c&-c \end{bsmallmatrix}$ as $c\rightarrow\infty$. The cactus network on the left has boundary vertices $v_1$ and $v_3$ shorted and $v_6$ and $v_7$ shorted.}
    \label{f:cactus}
\end{figure}

A \emph{pairing} $\tau$ on $[2n]$ is a fixed-point-free involution on $[2n]$. Given an electrical network $\Gamma$, there is a natural way to associate a pairing $\tau(\Gamma)$ \cite{Lam2018}. One first constructs the \emph{medial graph} $G(\Gamma)$, which depends only on the underlying unweighted graph of $\Gamma$. Place the vertices $1,2,\dots 2n$ around the boundary of the disk so that vertex $v_i$ is between $2i-1$ and $2i$. Then, construct a vertex $w_e$ for each edge $e$. The vertices $w_e$ and $w_{e'}$ are adjacent if and only if $e$ and $e'$ are incident and share a face. Moreover, if $e$ is incident to some boundary vertex $v_i$, then $w_e$ is adjacent to both $2i-1$ and $2i$. Lastly, if $v_i$ is an isolated vertex, we connect $2i-1$ and $2i$ in the medial graph. Notice that each vertex $w_e$ is four-valent.

\begin{figure}[h]
\centering
\begin{tikzpicture}

\draw (-3,0) circle [radius=1.5cm];
\draw[thick] (-3.5,0) -- (-2.5,0);
\draw[thick] (-3.5,0) -- (-4.275,.775);
\draw[thick] (-3.5,0) -- (-4.275,-.775);
\draw[thick] (-2.5,0) -- (-1.725,.775);
\draw[thick] (-2.5,0) -- (-1.725,-.775);

\fill (-3.5,0) circle (2pt);
\fill (-2.5,0) circle (2pt);
\fill (-4.275,.775) circle (2pt);
\fill (-4.275,-.775) circle (2pt);
\fill (-1.725,.775) circle (2pt);
\fill (-1.725,-.775) circle (2pt);
\fill (3.5,0) circle (2pt);
\fill (2.5,0) circle (2pt);
\fill (4.275,.775) circle (2pt);
\fill (4.275,-.775) circle (2pt);
\fill (1.725,.775) circle (2pt);
\fill (1.725,-.775) circle (2pt);

\draw (-1.725,.775) node [xshift=.2cm, yshift=.2cm] {$v_1$};
\draw (-1.725,-.775) node [xshift=.2cm, yshift=-.2cm] {$v_2$};
\draw (-4.275,-.775) node [xshift=-.2cm, yshift=-.2cm] {$v_3$};
\draw (-4.275,.775) node [xshift=-.2cm, yshift=.2cm] {$v_4$};

\draw (4,1.1) node [above] {1};
\draw (4.45,.35) node [right] {2};
\draw (4.45,-.35) node [right] {3};
\draw (4,-1.1) node [below] {4};
\draw (2,-1.1) node [below] {5};
\draw (1.505,-.35) node [left] {6};
\draw (1.505,.35) node [left] {7};
\draw (2,1.1) node [above] {8};
\draw (4.275,.775) node [xshift=.2cm, yshift=.2cm] {$v_1$};
\draw (4.275,-.775) node [xshift=.2cm, yshift=-.2cm] {$v_2$};
\draw (1.725,-.775) node [xshift=-.2cm, yshift=-.2cm] {$v_3$};
\draw (1.725,.7775) node [xshift=-.2cm, yshift=.2cm] {$v_4$};

\draw (3,0) circle [radius=1.5cm];
\draw[thick] (3.5,0) -- (2.5,0);
\draw[thick] (3.5,0) -- (4.275,.775);
\draw[thick] (3.5,0) -- (4.275,-.775);
\draw[thick] (2.5,0) -- (1.725,.775);
\draw[thick] (2.5,0) -- (1.725,-.775);
\draw[thick, dotted] (3,0) -- (3.85,.35);
\draw[thick, dotted] (3,0) -- (3.85,-.35);
\draw[thick, dotted] (3.85,.35) -- (3.85,-.35);
\draw[thick, dotted] (3,0) -- (2.15,.35);
\draw[thick, dotted] (3,0) -- (2.15,-.35);
\draw[thick, dotted] (2.15,.35) -- (2.15,-.35);
\draw[thick, dotted] (3.85,.35) -- (4,1.1);
\draw[thick, dotted] (3.85,-.35) -- (4,-1.1);
\draw[thick, dotted] (2.15,.35) -- (2, 1.1);
\draw[thick, dotted] (2.15,-.35) -- (2,-1.1);
\draw[thick, dotted] (3.85,.35) -- (4.45,.35);
\draw[thick, dotted] (3.85,-.35) -- (4.45,-.35);
\draw[thick, dotted] (2.15,.35) -- (1.505,.35);
\draw[thick, dotted] (2.15,-.35) -- (1.505,-.35);

\end{tikzpicture}
\caption{An electrical network (edge weights omitted) and its medial graph. Its medial pairing is $(1,4),(2,6),(3,7),(5,8)$.}
\label{f:medialgraph}
\end{figure}
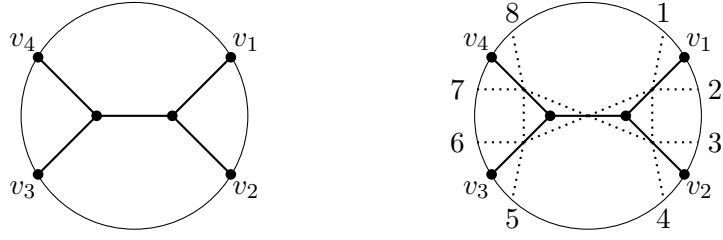

A \emph{strand} of $G(\Gamma)$ is a maximal sequence of edges that proceeds straight through each four-valent vertex. Strands either form a cycle in the interior of the disk or connect two (distinct) boundary vertices $i,j$; in the case of the latter, we refer to the strand as either ``strand $(i,j)$'' or simply ``strand $i$''. The medial graph of $\Gamma$ induces a \emph{medial pairing} $\tau(\Gamma)$ on $[2n]$ given by setting $\tau(i)=j$ for each strand $(i,j)$ between boundary vertices. This procedure naturally extends to cactus networks and the resulting pairing is invariant under electrical equivalence, even though the medial graph might change \cite{CIM, Lam2018}.

\begin{defn}[\bf Electroid Cell]
    The \emph{electroid cell} $E_\tau$ is the set of electrical equivalence classes of cactus networks $[\Gamma]\in E_n$  with the same medial pairing $\tau(\Gamma)=\tau$. 
\end{defn}

We say that the medial graph is \emph{critical} if no strand is an interior cycle, no strand is self-intersecting, and no two strands cross more than once. Let the \emph{crossing number} $c(\tau)$ denote the number of strand crossings in a critical medial graph with medial pairing $\tau$. A cactus network $\Gamma$ is \emph{critical} if its medial graph is critical, or equivalently, if it has $c(\tau(\Gamma))$ edges. Any cactus network is electrically equivalent to a critical cactus network \cite{Lam2018}.

\begin{defn}[\bf Poset of Pairings]\label{poset of pairings}
    The pairings on $[2n]$ form a natural poset graded by the crossing number $c(\tau)$ as follows. Any pair of crossing strands $(a,c)$ and $(b,d)$ with $a<b<c<d$ in cyclic order can be resolved in two ways. Namely, we can have $(a,b), (c,d)$ or $(a,d),(b,c)$. Let $\tau'<\tau$ if $\tau'$ can be obtained from $\tau$ by sequentially resolving crossings.
\end{defn}

Figure \ref{f:uncrossing} shows part of the poset for $n=3$.


\begin{thm} \cite{Lam2018} \label{ElectricalSummary}
The space of electrical equivalence classes of cactus networks is the disjoint union $E_n=\bigsqcup E_\tau$ of electroid cells indexed by pairings on $[2n]$. The analytic closure of any electroid cell is the disjoint union $\overline{E_\tau}=\bigsqcup_{\tau'\leq \tau} E_{\tau'}$. 

Each electroid cell is an open ball of dimension $c(\tau)$. There exists a critical cactus network $\Gamma$ for any pairing $\tau$. Letting the resistances vary on a fixed critical cactus network gives a homeomorphism $\RR_{>0}^{\text{Edges}(\Gamma)} \simeq E_{\tau(\Gamma)}$. 

\end{thm}

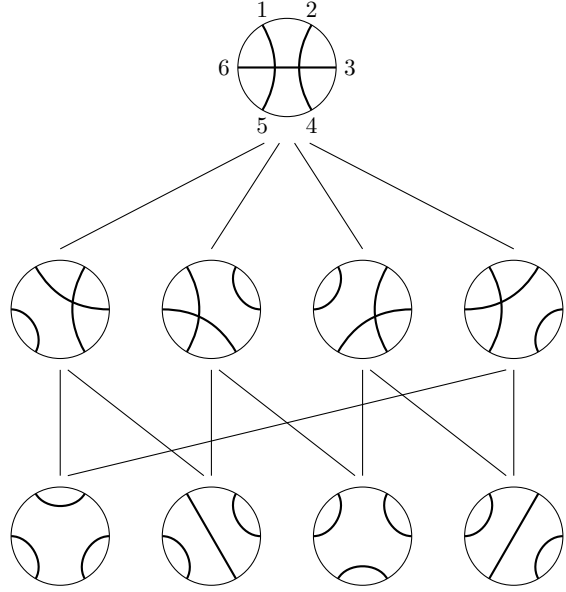
\begin{figure}
    \centering
    \begin{tikzpicture}
        \coordinate (A1) at (0,.2);
        \coordinate (B1) at (-3,-3);
        \coordinate (B2) at (-1,-3);
        \coordinate (B3) at (1,-3);
        \coordinate (B4) at (3,-3);
        \coordinate (C1) at (-3,-6);
        \coordinate (C2) at (-1,-6);
        \coordinate (C3) at (1,-6);
        \coordinate (C4) at (3,-6);

        \draw (A1) circle (.65);
        \draw (B1) circle (.65);
        \draw (B2) circle (.65);
        \draw (B3) circle (.65);
        \draw (B4) circle (.65);
        \draw (C1) circle (.65);
        \draw (C2) circle (.65);
        \draw (C3) circle (.65);
        \draw (C4) circle (.65);

        \draw (-.1,-.8) -- (-1, -2.2);
        \draw (.1,-.8) -- (1, -2.2);
        \draw (-.3,-.8) -- (-3, -2.2);
        \draw (.3,-.8) -- (3, -2.2);
        \draw (-3, -3.8) -- (-3, -5.2);
        \draw (-1, -3.8) -- (-1, -5.2);
        \draw (1, -3.8) -- (1, -5.2);
        \draw (3, -3.8) -- (3, -5.2);

        \draw (-2.9, -3.8) -- (-1.1, -5.2);
        \draw (-.9, -3.8) -- (0.9, -5.2);
        \draw (1.1, -3.8) -- (2.9, -5.2);
        \draw (2.9, -3.8) -- (-2.9, -5.2);

        \draw [shift={(0,.2)}, thick] (-.65,0) to[out=0,in=-180] (.65,0);
        \draw [shift={(0,.2)}, thick] (.325,.563) to[out=-120,in=120] (.325,-.563);
        \draw [shift={(0,.2)}, thick] (-.325,.563) to[out=-60,in=60] (-.325,-.563);

        \draw (A1) ++(-.325,.563) node [above,scale=.8] {1};
        \draw (A1) ++(.325,.563) node [above,scale=.8] {2};
        \draw (A1) ++(.65,0) node [right,scale=.8] {3};
        \draw (A1) ++(.325,-.563) node [below,scale=.8] {4};
        \draw (A1) ++(-.325,-.563) node [below,scale=.8] {5};
        \draw (A1) ++(-.65,0) node [left,scale=.8] {6};

        \draw [shift={(-3,-3)}, thick] (-.65,0) to[out=0,in=60] (-.325,-.563);
        \draw [shift={(-3,-3)}, thick] (.325,.563) to[out=-120,in=120] (.325,-.563);
        \draw [shift={(-3,-3)}, thick] (-.325,.563) to[out=-60,in=-180] (.65,0);

        \draw [shift={(-1,-3)}, thick] (.325,.563) to[out=-120,in=-180] (.65,0);
        \draw [shift={(-1,-3)}, thick] (-.65,0) to[out=0,in=120] (.325,-.563);
        \draw [shift={(-1,-3)}, thick] (-.325,.563) to[out=-60,in=60] (-.325,-.563);

        \draw [shift={(1,-3)}, thick] (-.65,0) to[out=0,in=-60] (-.325,.563);
        \draw [shift={(1,-3)}, thick] (.325,.563) to[out=-120,in=120] (.325,-.563);
        \draw [shift={(1,-3)}, thick] (.65,0) to[out=-180,in=60] (-.325,-.563);

        \draw [shift={(3,-3)}, thick] (-.65,0) to[out=0,in=-120] (.325,.563);
        \draw [shift={(3,-3)}, thick] (.65,0) to[out=-180,in=120] (.325,-.563);
        \draw [shift={(3,-3)}, thick] (-.325,.563) to[out=-60,in=60] (-.325,-.563);

        \draw [shift={(-3,-6)}, thick] (-.65,0) to[out=0,in=60] (-.325,-.563);
        \draw [shift={(-3,-6)}, thick] (.325,.563) to[out=-120,in=-60] (-.325,.563);
        \draw [shift={(-3,-6)}, thick] (.325,-.563) to[out=120,in=-180] (.65,0);

        \draw [shift={(-1,-6)}, thick] (.65,0) to[out=-180,in=-120] (.325,.563);
        \draw [shift={(-1,-6)}, thick] (-.65,0) to[out=0,in=60] (-.325,-.563);
        \draw [shift={(-1,-6)}, thick] (.325,-.563) to[out=120,in=-60] (-.325,.563);

        \draw [shift={(1,-6)}, thick] (.325,.563) to[out=-120,in=-180] (.65,0);
        \draw [shift={(1,-6)}, thick] (-.65,0) to[out=0,in=-60] (-.325,.563);
        \draw [shift={(1,-6)}, thick] (.325,-.563) to[out=120,in=60] (-.325,-.563);

        \draw [shift={(3,-6)}, thick] (-.65,0) to[out=0,in=-60] (-.325,.563);
        \draw [shift={(3,-6)}, thick] (.325,.563) to[out=-120,in=60] (-.325,-.563);
        \draw [shift={(3,-6)}, thick] (.65,0) to[out=-180,in=120] (.325,-.563);

    \end{tikzpicture}
    \caption{The lower order ideal of pairings on $[6]$ below $\tau=(1,5), (2,4), (3,6)$.}
    \label{f:uncrossing}
\end{figure}

\subsection{The Grassmannian and positroid varieties}
\label{s:emb into Gr}
Let $V$ be a vector space and let $0 \leq k \leq \dim V$. The Grassmannian $\Gr(k, V)$ is the space of $k$-dimensional subspaces of $V$. If our ground field $\FF$ is understood from context, we'll write $\Gr(k, m)$ as shorthand for $\Gr(k, \FF^m)$. We can represent a point of $\Gr(k, m)$ as the row span of a full rank $k \times m$ matrix. In this representation, the $\binom{m}{k}$ maximal minors of this matrix form the homogeneous Pl\"ucker coordinates on $\Gr(k,m)$. 

The totally nonnegative points of a projective variety $X$ are the locus of its real points where all Pl\"ucker coordinates are nonnegative, denoted by $X_{\geq0}$. In particular, when $X$ is a subvariety of $\Gr(k,m)$, we treat $X$ as a projective variety via the Pl\"ucker embedding.

\begin{defn}[\bf Grassmann necklace]
A \emph{$(k,m)$-Grassmann necklace} is an $m$-tuple $\mathcal{I} = (I_1, I_2, \dots, I_m)$ of subsets $I_a\in\binom{[m]}{k}$ such that 
\begin{enumerate}
    \item $I_{a+1} = I_a$ if $a\not\in I_a$.
    \item $I_{a+1} = I_a\setminus\{a\}\cup \{a'\}$ if $a\in I_a$.
\end{enumerate}    
\end{defn}
For example, $\mathcal{I}=(124,234,134,124)$ is a $(3,4)$-Grassmann necklace. Let $\leq_a$ be the cyclically shifted order on $[m]$ given by $a < a+1 < \dots < a-1$ taken modulo $m$. Then, $\leq_a$ extends naturally to $k$-subsets of $[m]$ using the lexicographic order with respect to $\leq_a$. Let $S$ denote the $k$-subsets of $[m]$ which are less than $I_a$ with respect to $\leq_a$ for some $a$. We define the \emph{open positroid variety} $\Pio_\mathcal{I}$ by
$$\mathring{\Pi}_{\mathcal{I}} := \{X\in \Gr(k,m) \mid \Delta_I(X)\not=0 \phantom{a}\forall I\in\mathcal{I} \phantom{a}\text{and}\phantom{a} \Delta_I(X)=0 \phantom{a} \forall I \in S\},$$ and the \emph{closed positroid variety} $\Pi_\cI$ {by $$\Pi_{\mathcal{I}} := \{X\in \Gr(k,m) \mid  \Delta_I(X)=0 \phantom{a} \forall I \in S\}.$$}
\begin{thm}\cite{KLS}
The Grassmannian $\Gr(k,m)$ is stratified by the open positroid varieties. Furthermore, we can recover the positroid cells by taking the totally nonnegative points of the corresponding open positroid variety. The closed positroid variety is the Zariski closure of the positroid cell.
\end{thm}

Positroid varieties can equivalently be indexed by bounded affine permutations. 

\begin{defn}[\bf Bounded affine permutation]
A \emph{bounded affine permutation of type-$(k,m)$} is a bijection $f:\mathbb{Z}\to\mathbb{Z}$ satisfying:
\begin{enumerate}
    \item $i\leq f(i)\leq i+m$
    \item $f(i+m)= f(i)+m$ for all $i\in \mathbb{Z}$
    \item $\sum_{i=1}^m(f(i)-i)=km$
\end{enumerate}
Let $B(k,m)$ denote the set of bounded affine permutations of type-$(k,m)$. 
\end{defn}

Given $X\in \Gr(k,m)$, define $f_X$ by 
$$f_X(i):=\min\{j\geq i\mid v_i\in\text{span}\{v_{i+1},v_{i+2}, \dots, v_j\}\}$$ 
where $v_i$ are the columns of a matrix representative of $X$, and the columns are extended periodically by $v_{i+m}:=v_i$. The map $f_X$ is guaranteed to be a bounded affine permutation of type-$(k,m)$ \cite{KLS}. One can equivalently define open positroid varieties by
$$\mathring{\Pi_f} :=\{X\in \Gr(k,m) \mid f_X = f\}.$$ 

Given $f\in B(k,m)$, the corresponding $(k,m)$-Grassmann necklace $\mathcal{I}(f)=(I_1, \dots, I_m)$ with $\Pio_f = \Pio_{\mathcal{I}(f)}$ is given by $$I_i = \{f(j) \mid j<i\text{ and } f(j)\geq i\} \text{ $\mod m$}.$$
Recall that the notion of stratification induces a poset structure on strata given by closure containment. This induces a poset structure on positroids, or equivalently, Grassmann necklaces and bounded affine permutations. We refer the reader to \cite{P} for details.

\subsection{Embedding $E_n$ into the Grassmannian}

In \cite{Lam2018}, Lam proves Theorem~\ref{ElectricalSummary} by embedding $E_n$ into the totally nonnegative Grassmannian $$i: E_n \hookrightarrow \Gr(n-1, {2n})_{\geq0}$$ as a linear slice. We describe this embedding in this subsection and the next. We first embed $E_n$ in some projective space.

Given an cactus network $\Gamma$, a \emph{grove} $F$ on $\Gamma$ is a spanning subforest such that each connected component contains at least one boundary vertex. The components naturally induce a non-crossing partition $\sigma(F)$ whose parts correspond to boundary vertices which lie in the same component. Note that if boundary vertices $i$ and $j$ are shorted in a cactus network, then $i,j$ will always be in the same block of $\sigma(F)$. Let $NC_n$ denote the set of all non-crossing partitions on $n$. The \emph{grove coordinates} $(L_\sigma(\Gamma))_{\sigma\in NC_n}$ are given by $$L_\sigma(\Gamma):= \sum\limits_{F: \sigma(F)=\sigma} \wt(F)$$ where the sum is over all groves of $\Gamma$ with boundary partition $\sigma$ and $\wt(F)$ is given by the product of the edge weights used in $F$. Moreover, the grove coordinates provide an embedding of $E_n$ into $\mathbb{P}^{NC_n-1}$ via $\Gamma\mapsto (L_\sigma(\Gamma))_{\sigma\in NC_n}$ \cite{Lam2018}.

We now describe a linear transformation from $\mathbb{P}^{NC_n-1}$ indexed by grove coordinates to $\mathbb{P}^{\binom{[2n]}{n-1}-1}$ indexed by Pl\"ucker coordinates of $\Gr(n-1,2n)$ \cite{Lam2018}. Place the numbers 1 to $2n$ around the boundary of a disk and identify boundary vertex $i$ with the number $2i-1$. A non-crossing partition $\sigma$ on the boundary vertices induces dual non-crossing partitions $\overline{\sigma}$ on $\{1,3,5, \dots, 2n-1\}$ and $\tilde{\sigma}$ on $\{2,4,6, \dots, 2n\}$, where $2i-1\in \overline{\sigma}_k$ if and only if $i\in\sigma_k$.

Given $\sigma\in NC_n$ and $I\in\binom{[2n]}{n-1}$, we say that $\sigma$ and $I$ are \emph{concordant}, denoted $\sigma \parallel I$, if each block of $\tilde{\sigma}$ and each block of $\overline{\sigma}$ has exactly one element not in $I$.

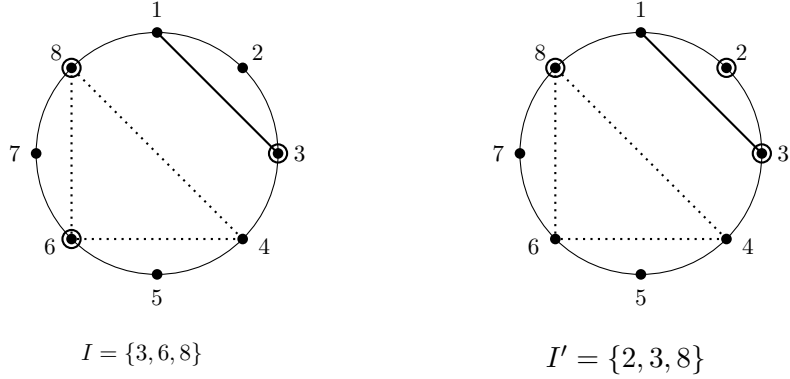
\begin{figure}[h]
\centering
\begin{tikzpicture}[scale=.4]
\coordinate (P) at (-8,0);
\coordinate (A) at (-8,4);
\coordinate (B) at (-5.17,2.83);
\coordinate (C) at (-4,0);
\coordinate (D) at (-5.17, -2.83);
\coordinate (E) at (-8, -4);
\coordinate (F) at (-10.83, -2.83);
\coordinate (G) at (-12,0);
\coordinate (H) at (-10.83, 2.83);
\fill (A) circle (5pt) node[above, yshift =.1cm, scale=.8]{1};
\fill (B) circle (5pt) node[above, xshift =.2cm, scale=.8]{2};
\fill (C) circle (5pt) node[right, xshift =.1cm, scale=.8]{3};
\fill (D) circle (5pt) node[right, xshift = .1cm, yshift = -.1cm, scale=.8]{4};
\fill (E) circle (5pt) node[below, yshift =-.1cm, scale=.8]{5};
\fill (F) circle (5pt) node[left, xshift = -.1cm, yshift =-.1cm, scale=.8]{6};
\fill (G) circle (5pt) node[left, xshift =-.1cm, scale=.8]{7};
\fill (H) circle (5pt) node[above, xshift = -.2cm, scale=.8]{8};
\draw (P) circle [radius=4cm] node[below, yshift=-2.4cm, xshift=-.2cm, scale=.8]{$I = \{3,6,8\}$};
\draw[thick] (A) -- (C);
\draw[thick, dotted] (D) -- (F);
\draw[thick, dotted] (H) -- (F);
\draw[thick, dotted] (D) -- (H);
\draw[thick] (C) circle [radius=.3cm];
\draw[thick] (F) circle [radius=.3cm];
\draw[thick] (H) circle [radius=.3cm];

\coordinate (P1) at (8,0);
\coordinate (A1) at (8,4);
\coordinate (B1) at (10.83,2.83);
\coordinate (C1) at (12,0);
\coordinate (D1) at (10.83, -2.83);
\coordinate (E1) at (8, -4);
\coordinate (F1) at (5.17, -2.83);
\coordinate (G1) at (4,0);
\coordinate (H1) at (5.17, 2.83);
\fill (A1) circle (5pt) node[above, yshift =.1cm, scale=.8]{1};
\fill (B1) circle (5pt) node[above, xshift =.2cm, scale=.8]{2};
\fill (C1) circle (5pt) node[right, xshift =.1cm, scale=.8]{3};
\fill (D1) circle (5pt) node[right, xshift = .1cm, yshift =-.1cm, scale=.8]{4};
\fill (E1) circle (5pt) node[below, yshift =-.1cm, scale=.8]{5};
\fill (F1) circle (5pt) node[left, xshift = -.1cm, yshift =-.1cm, scale=.8]{6};
\fill (G1) circle (5pt) node[left, xshift =-.1cm, scale=.8]{7};
\fill (H1) circle (5pt) node[above, xshift = -.2cm, scale=.8]{8};
\draw (P1) circle [radius=4cm] node[below, yshift=-2.4cm, xshift=-.2cm]{$I' = \{2,3,8\}$};;
\draw[thick] (A1) -- (C1);
\draw[thick, dotted] (D1) -- (F1);
\draw[thick, dotted] (H1) -- (F1);
\draw[thick, dotted] (D1) -- (H1);
\draw[thick] (C1) circle [radius=.3cm];
\draw[thick] (B1) circle [radius=.3cm];
\draw[thick] (H1) circle [radius=.3cm];

\end{tikzpicture} 
\caption{The non-crossing partition $\sigma=12|3|4$ induces the dual non-crossing partitions $\ov{\sigma}=13|5|7$ and $\tilde{\sigma}=2|468$. The subset $I=\{3,6,8\}$ is concordant with $\sigma$ but $I'=\{2,3,8\}$ is not.}
\label{f:concordant}
\end{figure}

Define a linear map $\mathbb{P}^{NC_n-1}\rightarrow \mathbb{P}^{\binom{[2n]}{n-1}-1}$ using  a matrix $H=(a_{I\sigma})$ with columns indexed by non-crossing partitions on $[n]$ and rows indexed by $(n-1)$-subsets of $[2n]$ as follows. 
$$a_{I\sigma}=
\begin{cases}
    1,\qquad \sigma \parallel I\\
    0,\qquad else
\end{cases}$$

\begin{thm}\cite{Lam2018} \label{embedding electroid cell in Grassmannian}
The map $E_n\hookrightarrow \mathbb{P}^{NC_n-1}\rightarrow \mathbb{P}^{\binom{[2n]}{n-1}-1}$ given by composing the grove coordinate map with the linear map $H$ gives an embedding 
$$i: E_n\hookrightarrow \Gr(n-1,2n)_{\geq0}.$$
\label{t:En emb}
\end{thm}

In particular, for any cactus network $\Gamma$, the Pl\"ucker coordinates of the point in $\Gr(n-1,2n)_{\geq0}$ represented by the network can be written as a sum of grove coordinates
$$\Delta_I(i(\Gamma)) = \sum_{\sigma \parallel I} L_\sigma(\Gamma).$$
\begin{defn}\cite{Lam2018}
The \emph{electroid space} $\chi_n$ is the scheme theoretic intersection between the linear subspace $\mathbb{P}^{NC_n-1}\hookrightarrow \mathbb{P}^{\binom{[2n]}{n-1}-1}$ and the Grassmannian $\Gr(n-1,2n)\hookrightarrow \mathbb{P}^{\binom{[2n]}{n-1}-1}$. 
\end{defn}

It is clear that $i(E_n)$ is contained in the real points of $\chi_n$. It can be seen from the proof of \cite[Theorem 1.5]{CGS2021} that the Zariski closure of $i(E_n)$ is set theoretically equal to $\chi_n$. It turns out that they are equal as a scheme. We attribute the following theorem to \cite{Gao_2024}, as it can be derived as an immediate consequence of results therein, which we establish below.

\begin{thm}\cite{Gao_2024}\label{two notions of electroid space}
   The electroid space $\chi_n$ equals the Zariski closure of $i(E_n)$ as a scheme. 
\end{thm}

\begin{proof}
    Fix coordinates $(\Delta_I:I\in \binom{[2n]}{n-1})$  on $\mathbb{P}^{\binom{[2n]}{n-1}-1}$ and coordinates $(L_\sigma:\sigma \in NC_n)$ on $\mathbb{P}^{NC_n-1}$. Let $\mathcal{J}$ be the Pl\"ucker ideal in $k[\Delta_I]$ and $\overline{\mathcal{J}}$ be the corresponding ideal in the quotient ring $k[\Delta_I]/\mathcal{I}$. Let $\mathcal{I}$ be the linear ideal that defines the linear subspace $\mathbb{P}^{NC_n-1}\hookrightarrow \mathbb{P}^{\binom{[2n]}{n-1}-1}$ embedded via the matrix $H=(a_{I\sigma})$. Thus, $\mathcal{I}$ is the defining ideal for $\chi_n$ as a subscheme of $\Gr(n-1,2n)$. 
    
    Consider the ring isomorphism
    $k[\Delta_I]/\mathcal{I}\simeq   k[L_\sigma]$ via $\Delta_I\mapsto \sum_{\sigma \parallel I} L_\sigma$.  Let $\mathcal{K}$ be the ideal in $k[L_\sigma]$ arising as the image of $\overline{\mathcal{J}}$ under the ring isomorphism. We thus have an isomorphism $k[\Delta_I]/(\mathcal{I}+\mathcal{J})\simeq k[L_\sigma]/\mathcal{K}$. In \cite{Gao_2024}, they define the grove algebra $G_n$ as the ring $k[L_\sigma]/\mathcal{K'}$, where $\mathcal{K}'$ is the ideal of all polynomials in $L_\sigma$ that vanishes on all cactus networks. By \cite[Theorem 6.12]{Gao_2024}, we have equality $\mathcal{K}=\mathcal{K}'$ 
    
    Observe that ${\mathcal{I}+\mathcal{J}}$, being the isomorphic image of $\mathcal{K}=\mathcal{K}'$ under the ring isomorphism $k[\Delta_I]/(\mathcal{I}+\mathcal{J})\simeq   k[L_\sigma]/\mathcal{K}$, is thus the ideal of all polynomials in $\Delta_I$ that vanishes on all cactus networks, where the evaluation of $\Delta_I$ on a cactus network is prescribed by the ring isomorphism, given explicitly by $\Delta_I(\Gamma)= \sum_{\sigma \parallel I} L_\sigma(\Gamma)$. Notice that this is the same map as the embedding $E_n\hookrightarrow \Gr(n-1,2n)_{\geq0}\hookrightarrow \mathbb{P}^{\binom{[2n]}{n-1}-1}$ of the space of cactus networks in Theorem \ref{embedding electroid cell in Grassmannian}. Thus, ${\mathcal{I}+\mathcal{J}}$ is the ideal of all polynomials in $\Delta_I$ that vanishes on the image of $E_n\hookrightarrow \Gr(n-1,2n)_{\geq0}\hookrightarrow \mathbb{P}^{\binom{[2n]}{n-1}-1}$. This embedding factors through the Grassmannian and, indeed, the ideal contains the Pl\"ucker ideal. Working inside the Grassmannian, $\mathcal{I}$ is the ideal of all polynomials in Pl\"ucker coordinates that vanishes on the image $i:E_n\hookrightarrow \Gr(n-1,2n)_{\geq0}$. In other words, $\mathcal{I}$ is the homogeneous ideal for the Zariski closure of $i(E_n)$ in $\Gr(n-1,2n)$. 
\end{proof}

\subsection{Electroid varieties}\label{subsection:Electroid varieties}
We define a specific subset of bounded affine permutations as follows.
\begin{defn}\cite{Lam2018}
An \emph{electrical affine permutation} is a type-$(n-1,2n)$ bounded affine permutation $f$ such that, defining $g(i):=f(i)+1$, we have $i+1 \leq g(i) \leq i +2n -1$ and $g(g(i))=i+2n$.
\end{defn}

Let $\El(n)$ denote the set of all electrical affine permutations. The poset structure of affine bounded permutations defines a poset structure on $\El(n)$. Recall from Definition \ref{poset of pairings} that pairings on $[2n]$ form a poset.

\begin{thm}\label{matching to permutation}\cite{Lam2018}
    There is a poset isomoprhism between $\El(n)$ and the poset of pairings on $[2n]$ via $\tau\mapsto f_{\tau}$, where $f_\tau(i) := \tau(i)-1$.
\end{thm}
 
Recall that the poset of bounded affine permutations is isomorphic to the poset of Grassmann necklaces. We will describe the explicit bijection from pairings to Grassmann necklaces in Lemma \ref{pairing to necklace}.

\begin{thm} \label{cell decomp} \cite{Lam2018}
The following are equivalent.
\begin{itemize}
    \item The intersection between $\Xn $ and the open positroid variety $ \Pio_f$ is nonempty.
    \item The intersection between $\Xn$ and the positroid cell $ (\Pio_{f})_{\geq0}$ is nonempty.
    \item $f\in \El(n)$.
\end{itemize}
For $f=f_\tau \in \El(n)$, the intersection $\chi_n\cap (\Pio_{f_{\tau}})_{\geq0}$ is precisely the image $i(E_\tau)$.
\end{thm}
In particular, the above Theorem shows that the only totally nonnegative points of the Zariski closure of $i(E_n)$ are given by $i(E_n)$ itself.
\begin{defn}\cite{Lam2018}
    The \emph{open} and \emph{closed electroid varieties} $\Xo_{\tau}$ and $\chi_\tau$ are defined by $$\Xo_\tau:=\Pio_{f_\tau} \cap \Xn,\qquad \chi_\tau:= \Pi_{f_\tau} \cap \Xn.$$  
    Recall that $E_\tau$ was called an electroid cell. We also call its isomorphic image $i(E_\tau)=(\Xo_\tau)_{\geq0}$
an electroid cell.\end{defn}

\begin{rem}
    Positroid cells got their name from combining ``positive" and ``matroid." In the stratification of the totally nonnegative Grassmannian $\Gr(k,m)_{\geq0}$ into positroid cells, two points in the totally nonnegative Grassmannian are in the same cell if and only if the same set of Pl\"ucker coordinates vanish, giving the matroid data. Analogously, two points in $(\chi_n)_{\geq0}=i(E_n)$ are in the same electroid cell if and only if after embedding to $\mathbb{P}^{NC_n-1}$, the same set of grove coordinates vanish. The above theorem can be interpreted as saying that, with respect to the linear map $H$, the data of the set of vanishing grove coordinates gives the same information as the data of the set of vanishing Pl\"ucker coordinates. The following theorem makes this correspondence explicit.
\end{rem}
\begin{thm}\cite[Prop 4.4 and Prop 4.10]{Lam2018}
Given a critical cactus network $\Gamma$, define the \emph{electroid} $\mathcal{E}(\Gamma) := \{\sigma \mid L_\sigma(\Gamma)\not=0\}$. The electroid is constant on each electroid cell $E_\tau$. That is, for any critical cactus network $\Gamma$ with $\tau(\Gamma)=\tau$, the electroid $\mathcal{E}(\Gamma)$ is the same. 
\label{vanishing groves}
\end{thm}
Since open positroid varieties form a partition of the Grassmannian, the above theorems show that open electroid varieties form a partition of $\Xn$ whose totally nonnegative points form a stratification for ${(\Xn)}_{\geq 0}$. However, up until now, little was known about their geometry. It is not even a priori clear that $\chi_\tau$ is the Zariski closure of $\Xo_\tau$.

\subsection{Cyclic Symmetry on Electroid Spaces}\label{cyclic symmetry}
Let the cyclic group $C_n$ acts on the Grassmannian $\Gr(k,n)$ via $(v_1,\cdots , v_n)\mapsto ((-1)^k v_{n}, v_1,\cdots, v_{n-1})$, where the $v_i$'s are the columns of a matrix representative of a point. It is clear that the totally nonnegative part $\Gr(k,n)_{\geq 0}$ is closed under this action (see, for example,~\cite[Remark 3.3]{P}). 

In fact, the electroid space is also closed under this action. To see this, it suffices to show it for $i(E_n)$, because $\chi_n$ is the Zariski closure of $i(E_n)$. Recall that $i(E_n)$ is all points represented by a cactus network. It is shown in Section 2.5 of \cite{CGS2021} that applying a cyclic shift to a point represented by a cactus network gives the point represented by the dual cactus network.

One can combinatorially show that the medial pairing given by the dual cactus network is a cyclic shift of the medial pairing of the original cactus network. Here, we provide a roundabout way. For an affine bounded permutation $f$, define the cyclically shifted affine bounded permutation $C(f)$ by $C(f)(i)=f(i-1)+1$. It is well-known, see in \cite{KLS} or, more explicitly, \cite[Lemma 3.12]{AGH2024}, that the cyclic action on $\Gr(k,n)$ restricts to an isomorphism between closed electroid varieties $\Pi_f\to \Pi_{C(f)}$. 

Now, from Theorem \ref{matching to permutation}, it is straightforward to see that $C(f_{\tau}) = f_{\tau'}$, where $\tau'$ is the cyclic shift of $\tau$ given by $\tau'(i+1)=\tau(i)+1$. Thus, the cyclic shift operator action on $\Gr(k,n)$ permutes the electroid varieties by cyclically rotating the corresponding medial pairing.

\subsection{Grassmannian Duality \label{Grassmannian Duality}}
It is well known that $\Gr(k,n)$ is isomorphic to $\Gr(n-k,n)$. In particular, a subvariety $\chi_n\hookrightarrow \Gr(n-1,2n)$ can naturally be embedded in $\Gr(n+1,2n)$ via the duality isomorphism.

We fix a specific such isomorphism $\Gr(n-1,2n)\to \Gr(n+1,2n)$ by $X\mapsto \tilde{X}$ where $\Delta_I(\tilde{X})=\Delta_{[2n]\setminus I}(X)$ for each $I\in \binom{[2n]}{n+1}$. Note that this isomorphism, compared to the more standard one given by taking orthogonal complement, is off by a sign change map which switches signs of even columns of a matrix representative, see \cite[Section 7]{Hochster75} and \cite[Theorem 2.2.8]{oxley06}. Our choice of the duality has the property that $\Gr(n-1,2n)_{\geq0}$ is mapped to $\Gr(n+1,2n)_{\geq0}$. Thus, we can {also} think of $\chi_n\subset \Gr(n+1,2n)$ with totally nonnegative points given by the image $i(E_n)$.

The convention of embedding $\chi_n$ in $\Gr(n-1,2n)$ is adopted in \cite{Lam2018}, whereas the convention of embedding in $\Gr(n+1,2n)$ is adopted in \cite{CGS2021}. We will use both conventions, since we need results from both the above two papers.

\subsection{The Lagrangian Grassmannian} \label{The Lagrangian Grassmannian}

In 2021, \cite{BGKT} and \cite{CGS2021} gave another reason to care about the electroid space, showing that it is abstractly isomorphic to the Lagrangian Grassmannian $\LG(n-1,2n-2)$. If $V$ is a vector space equipped with a skew-symmetric bilinear form $\Omega$, then we define a subspace $U$ of $V$ to be \emph{isotropic} if the restriction of $\Omega$ to $U$ is $0$. The space of isotropic $k$-planes in $V$ is denoted $\IG^{\Omega}(k, V)$. 
We define a skew-symmetric bilinear form on $\RR^{2n}$ by
\begin{equation*}
\Omega((x_1, \dots, x_{2n}), (y_1, \dots y_{2n})) = \sum_{i=1}^{2n} (x_i y_{i+1} - x_{i+1} y_i)
\label{OmegaForm}
\end{equation*} 
with indices modulo $2n$. The bilinear form has a $2$-dimensional kernel, and thus its isotropic Grassmannian $\IG^{\Omega}(n+1, \RR^{2n})$ is isomorphic to $\LG(n-1,2n-2)$. 

We define a sign flipped version of the totally nonnegative Grassmannian; $\Gr(n+1,2n)_{\geq 0}^D$ is the image of $\Gr(n+1,2n)_{\geq 0}$ under the diagonal matrix $$\text{diag}(1,1,-1,-1,1,1,-1,-1,\ldots,(-1)^n, (-1)^n).$$ We write $(\Pio_f)^{D}_{\geq 0}$ for the corresponding sign flipped version of an open positroid variety. Define the totally nonnegative part of $\IG^{\Omega}(n+1, \RR^{2n})$ by the intersection $\Gr(n+1, 2n)_{\geq 0}^D \cap \IG^{\Omega}(n+1, \RR^{2n})$. Note that this gives a different totally nonnegative part of $\IG^{\Omega}(n+1, \RR^{2n})\simeq \LG(n-1,2n)$ than those considered by Lusztig, Rietsch, and Karpman \cite{Lusztig_1994, Rietsch_1999, Karpman_2018}.  

\begin{thm} \label{IsotropicSummary}
The image $i(E_n)$ in $\Gr(n+1,2n)$, after right multiplication by $D$, is given by $\Gr(n+1, 2n)_{\geq 0}^D \cap \IG^{\Omega}(n+1, \RR^{2n})$. Thus, its Zariski closure, after right multiplication by $D$, is $\chi_n=\IG^{\Omega}(n+1, 2n)\simeq \LG(n-1,2n-2)$.
\end{thm}

\subsection{Main results}

\begin{thm}\label{Main1}
The Lagrangian Grassmannian $\LG(n-1, 2n-2)$ admits a well-behaved stratification by open and closed electroid varieties, extending the stratification of its totally nonnegative part by electroid cells. More precisely, we have the following. 
\begin{enumerate}
    \item Open electroid varieties are irreducible, smooth, and have expected dimension given by the number of crossings in the corresponding strand diagram.
    
    \item Totally nonnegative points of an open electroid variety are Zariski dense in the open electroid variety and are exactly the corresponding electroid cell.
    
    \item Closed electroid varieties are reduced, irreducible, regular in codimension one, and have expected dimension.
    
    \item There is a Frobenius splitting on $\LG(n-1, 2n-2)$ under which the closed electroid varieties are compatibly split.
    
    \item Any closed electroid variety $\chi_\tau$ is the Zariski closure of its open electroid variety $\Xo_{\tau}$ and is a disjoint union $\chi_\tau=\bigsqcup_{\tau'\leq\tau}\Xo_{\tau'}$ of open electroid varieties. Thus, they form a stratification.

    \item There exists an isomorphism between $\chi_{n_1}\times \chi_{n_2}$ and a certain closed electroid variety of $\chi_{n_1+n_2}$, respecting the stratification on both sides by electroid varieties. 

    \item Any critical cactus network embeds an open ball $\mathbb{R}_{> 0}^\ell$ into the totally nonnegative points of an open electroid variety. This can be extended algebraically to embed an algebraic torus $(\mathbb{C}^*)^\ell$ as an open dense subscheme of the open electroid variety.
    
\end{enumerate} 
\end{thm}

\section{Irreducibility}
\label{s:irred}
In this section, we will work with the convention that electroid varieties are embedded in $\Gr(n-1,2n)$. 
The group $GL_m$ acts on $\Gr(k,m)$ by right multiplication. For $a\in\mathbb{C}$ and $1\leq i\leq m$, the element $x_i(a)\in GL_m$ is the elementary matrix differing from the identity matrix by the entry $a$ in the $i$-th row and $(i+1)$-st column (wrapping around to the $1$-st column if $i=m$). Similarly, $y_i(a)$ is $x_i(a)^T$. We put $u_i(a):=x_i(a)y_{i-1}(a) = y_{i-1}(a)x_i(a)$.

If $v_1$, $v_2$, \dots, $v_m$ are the columns of a matrix representative for a point of $\Gr(k,m)$, then the matrices $x_i(a)$, $y_i(a)$ and $u_i(a)$ act on these columns by
\[\begin{array}{lcl}
x_i(a) (v_1,  \ldots, v_{i-1}, v_i, v_{i+1}, \ldots, v_{2n}) &=& (v_1, v_2, \ldots, v_{i-1}, v_i, av_i + v_{i+1}, \ldots, v_{2n}) \\
y_i(a) (v_1, v_2, \ldots, v_{i-1}, v_i, v_{i+1}, \ldots, v_{2n}) &=& (v_1, v_2, \ldots, v_{i-1}, v_i+av_{i+1}, v_{i+1}, \ldots, v_{2n}) \\
u_i(a) (v_1, v_2, \ldots, v_{i-1}, v_i, v_{i+1}, \ldots, v_{2n}) &=& (v_1, v_2, \ldots, av_i + v_{i-1}, v_i, av_i + v_{i+1}, \ldots, v_{2n}) \\
\end{array}\]

\begin{rem} \label{uActionDual}
    If we had, instead, worked in $\Gr(n+1, 2n)$, then we would need to replace $x_i$ and $y_i$ by their transposes, and thus the last formula would be \[ u_i(a) (v_1, v_2, \ldots, v_{i-1}, v_i, v_{i+1}, \ldots, v_m) = (v_1, v_2, \ldots, v_{i-1}, av_{i-1} + v_i + a v_{i+1}, v_{i+1}, \ldots, v_m). \] 
\end{rem}

\subsection{Embedding induced by an $i$-crossing pair}
In this subsection, we introduce a key embedding relating a pair of open electroid varieties whose pairings are related by crossing adjacent strands, given in Proposition \ref{p:uncrossing adjacent}.

\begin{defn}[\bf $i$-crossing pair]
A pair of pairings $\tau\lessdot\tau'$ is called an \emph{$i$-crossing pair} if $\tau'$ is formed from $\tau$ by crossing distinct non-intersecting strands $i$ and $i+1$ such that $\tau'(i)=\tau(i+1)$ and $\tau'(i+1)=\tau(i)$. 
\end{defn}

Throughout this section, it will be useful to be able to manipulate sets in the Grassmann necklace in the following way. By Remark \ref{r: no i,i-1}, given a Grassmann necklace $\mathcal{I}=\mathcal{I}_{f_{\tau}}$ associated to a pairing as in Subsection \ref{subsection:Electroid varieties}, we always have $i\not\in I_{i+1}$.

\begin{defn}
Let $\mathcal{I}$ be a Grassmann necklace associated to a pairing. If $i+1\in I_{i+1}$, define $I_{i+1}':=I_{i+1}-\{i+1\}\cup\{i\}$. Then, $I_{i+1}'$ is again a subset of $[2n]$ of size $n-1$.
\end{defn}

\begin{prop}\label{p:uncrossing adjacent}
Let $\tau\lessdot\tau'$ be an $i$-crossing pair. We have an open embedding of algebraic varieties
    \begin{align*}
    \psi: \mathring{\chi}_{\tau}\times\mathbb{C}^* &\longrightarrow  \mathring{\chi}_{\tau'}\\
    (Y,a) &\longmapsto u_i(a).Y
\end{align*}
The image of $\psi$ is the dense open subscheme of $\mathring{\chi}_{\tau'}$ given by the nonvanishing of $\Delta_{I_{i+1}'(\tau')}$. The inverse of $\psi$ can be obtained by $X\mapsto (u_i(-a).X,a)$, where $a=\frac{\Delta_{I_{i+1}(\tau')}}{\Delta_{I_{i+1}'(\tau')}}$.
\end{prop}

To prove Proposition \ref{p:uncrossing adjacent}, we use the following series of lemmas. 

\begin{lem} Suppose $f\in \El(n)$ and $X\in\Xo_f$. Let $a\in\C^*$.
\begin{enumerate}
    \item If $i+1<f(i+1)<f(i)<i+n$, then $X'=x_i(a).X\in\mathring{\Pi}_{f'}$ where $f'=fs_i$. 
    \item If $i+1-n < f^{-1}(i+1) < f^{-1}<i$, then $X'=y_i(a).X\in\mathring{\Pi}_{f'}$ where $f'=s_if$.
    \end{enumerate}
\label{l:crossing}
\end{lem}

\begin{proof}
This proof follows the same structure as the proof of \cite[Proposition 3.10]{Lam2018}. We prove only (1), as the proof of (2) follows similarly. Let $v_i$ be the $i$-th column of a $k\times m$ matrix which represents $X$. Then, $X'$ can obtained from $X$ by replacing $v_{i+1}$ with $av_i+v_{i+1}$. Let $v_i'$ be the $i$-th column of the resulting matrix for $X'$. Our aim is to show that $f'=fs_i$. 

We have that $\spann(v_i)=\spann(v_i')$ and $\spann(v_i,v_{i+1}) = \spann(v_i',v_{i+1}')$. Thus, $f(j) = f'(j)$ for all $j\not\equiv i,i+1$ modulo $m$. 
    
Now, $v_{i+1}$ is in $\spann(v_{i+2},v_{i+3},\cdots ,v_{f(i+1)})$. So, $$\spann(v'_{i+1}=av_i + v_{i+1},v_{i+2}'=v_{i+2},\cdots ,v_{f(i+1)}'=v_{f(i+1)})$$ contains $\spann(av_i + v_{i+1},v_{i+1})$, which contains $\spann(v_i)$ because $a\not=0$. Thus, $v_i'=v_i\in\spann(v_{i+1}',\dots ,v_{f(i+1)}')$ and so $f'(i)\leq f(i+1)<f(i)$. Hence, $f'$ must be obtained from $f$ by swapping the values of $f(i)$ and $f(i+1)$ and so $f'=fs_i$.
\end{proof}

\begin{prop}\label{p:crossing}
    The electriod space $\chi$ is closed under the action of left multiplication by $u_i(a)$. Furthermore, let $\tau\lessdot\tau'$ be an $i$-crossing pair. For $X\in\Xo_\tau$ and $a\in\C^*$, 
    $$X':= u_i(a).X\in\Xo_{\tau'}.$$
\end{prop}

\begin{proof}
We apply Lemma \ref{l:crossing} twice together with the fact that $\chi_n$ is closed under the action of $u_i(a)$ by \cite[Proposition 5.15]{Lam2018}.
\end{proof}

\begin{rem}
    If the reader wants to check directly by hand that $u_i(a)$ preserves the Lagrangian Grassmannian, they should remember that the Lagrangian Grassmannian is embedded in $\Gr(n+1, 2n)$ and use the dual formula from Remark~\ref{uActionDual}.
\end{rem}

In Subsection \ref{subsection:Electroid varieties}, we established that electric permutations are in bijection with pairings via the map $\tau\mapsto f_{\tau}$. An electric permutation $f_{\tau}$ also corresponds to a Grassmann necklace $\mathcal{I}_{f_{\tau}}$. We now describe how to explicitly compute the Grassmann necklace $\mathcal{I}_\tau=\mathcal{I}_{f_{\tau}} = (I_1, \dots ,I_{2n})$, from the pairing $\tau$. 

\begin{lem}\label{pairing to necklace}
Fix $i$. Let $\tau$ be a pairing on $[2n]$. Then, $I_{i}(\tau)$ can be computed as follows. For each strand which is not the strand $(i,\tau(i))$, collect the endpoint which is reached first when moving clockwise from $i$. They form the set $I_i+1$ modulo $2n$, {where here ``modulo $2n$'' means we take representatives in $[2n]$. Reducing each element by $1$ yields $I_i$.}  
\end{lem}
\begin{proof}
    Given a bounded affine permutation $g\in B(k,n)$, the corresponding Grassmann necklace is given by $\mathcal{I}(g) = (I_1, I_2, \dots, I_n)$ where $I_i=\{g(a) \mid a<i \text{ and } g(a)\geq i\}$ mod $n$. 
    
    Let $f_\tau\in B(n-1,2n)$ be the electrical affine permutation corresponding to $\tau$. It suffices to prove that the result for $I_{2n}$, as symmetry extends the result to all $I_i$. Consider a strand $(a,b)$ which is not strand $(2n,\tau(2n))$. Let $a<b$; then, $a$ is reached first when moving clockwise from $2n$. Then, $f_\tau(b) = a-1 +2n \geq 2n$. Hence, for all $n-1$ strands $(a,b)$, we have $b<2n$ and $f_\tau(b) \geq 2n$, so $a-1 \equiv f_\tau(b)\in I_{2n}$. As $I_{2n}$ has only $n-1$ elements, these are precisely the elements of $I_{2n}$.
\end{proof}

\begin{rem}
Lemma \ref{pairing to necklace} can be reinterpreted as saying that for each strand which is not $(i,\tau(i))$, take the endpoint $b$ which is \emph{not} reached first when moving clockwise from $i$ and add $f(b)$ to $I_i$.
\end{rem}

\begin{rem}\label{r: no i,i-1}
Let $\mathcal{I}=\mathcal{I}_{\tau}$ be a Grassmann necklace coming from a pairing $\tau$. Lemma \ref{pairing to necklace} shows that $i$ and $i-1$ are never in $I_{i+1}$. Furthermore, $i+1\in I_{i+1}$ if and only if $\tau(i+1)\not=i+2$.
\end{rem}

As a corollary of Lemma \ref{pairing to necklace}, we have the following. 

\begin{lem}\label{technical}
   Let $\tau\lessdot\tau'$ be an $i$-crossing pair and $f$ and $f'$ the corresponding electrical affine permutations. We have that $f(i)>f(i+1)$ and $I_{i+1}(f')=I_{i+1}(f)-\{f(i)\}\cup \{f(i+1)\}$. 
\end{lem}
In particular, $I_{i+1}(f')$ is strictly lexicographically less than $I_{i+1}(f)$.
\begin{lem}
Given $X\in\chi_n$ and $a\in\C^*$, let $X':=u_i(a).X$. Then, for any $J\in\binom{[2n]}{n-1}$, 
$$\Delta_J(X')=
\begin{cases}
    \Delta_J(X) + a\Delta_{J-\{i-1\}\cup\{i\}}(X)+a\Delta_{J-\{i+1\}\cup\{i\}}(X)&i-1,i+1\in J\text{ and } i\not\in J\\
    \Delta_J(X) +a\Delta_{J-\{i-1\}\cup\{i\}}(X)&i-1\in J\text{ and } i,i+1\not\in J\\
    \Delta_J(X) +a\Delta_{J-\{i+1\}\cup\{i\}}(X)&i+1\in J\text{ and } i-1,i\not\in J\\
    \Delta_J(X)&\text{else.}
\end{cases}$$
\label{l:pluckerchange}
\end{lem}
\begin{proof}
This follows immediately from the multilinearity (with respect to columns) of the determinant and the fact that $u_i(a).X$ is formed from $X$ by adding $a$ times column $i$ to columns $i-1$ and $i+1$.
\end{proof}

Below, we list several properties satisfied by Pl\"ucker coordinates on positroid and electroid varieties.
\begin{lem} \label{vanishing lem 1}
Let $f\in B(k,n)$ be an affine bounded permutation and let $i$ be an index for which  $f(i) \neq i$. Then for each $j <_{i+1} f(i)$, the coordinate $\Delta_{I_{i}(f)\setminus\{i\}\bigcup \{j\} }$ vanishes on the closed positroid variety ${\Pi}_{f}$.

Hence, let $f\in \El(n)$ an electrical permutation with $f(i)\not=i$. Then, for each $j
    <_{i+1}f(i)$, the coordinate $\Delta_{I_{i}(f)\setminus\{i\}\bigcup \{j\} }$ vanishes on the closed electroid variety ${\chi}_{f}$.
\end{lem}

\begin{proof}
Since $f(i)\not=i$, we have $i\in I_i(f)$ and $I_{i+1}(f)=I_{i}\setminus\{i\}\bigcup \{f(i)\}$. So for each $ j
    <_{i+1}f(i)$, $I_{i}\setminus \{i\}\bigcup \{j\}=I_{i+1}\setminus\{f(i)\}\bigcup \{j\}<_{i+1} I_{i+1}$. So $\Delta_{I_{i}\setminus \{i\}\bigcup \{j\}}=0$ on the closed positroid variety $\Pi_f$. The statement about electroid varieties follows from the fact that $\chi_f$ is the scheme theoretic intersection $\Pi_f\bigcap \chi$. 
\end{proof}

In the context of Proposition \ref{p:uncrossing adjacent}, the rational function $a=\frac{\Delta_{I_{i+1}}(X)}{\Delta_{I_{i+1}'}(X)}$ can be interpreted as the inverse of a coefficient of a linear expansion of columns. We present this result in the next lemma, which will be of use in a later section.

\begin{lem}\label{recover a}
Let $X\in \Xo_\tau\subset \Gr(k,n)$ be represented by a matrix with columns given by $v_1,v_2,\cdots,v_n$. Let $f$ be the corresponding electrical permutation such that $f(i)<f(i+1)$. Then there is a unique $\lambda$ for which $v_i-\lambda v_{i+1}\in \spann \{v_{i+2},\cdots, v_{f(i)}\}$; specifically, this unique $\lambda$ is $\frac{\Delta_{I_{i+1}'(\tau)}(X)}{\Delta_{I_{i+1}(\tau)}(X)}$. 
\end{lem}
\begin{proof}
By Lemma \ref{pairing to necklace}, $f(i)<f(i+1)$ implies that $f(i+1)\not=i+1$ or equivalently $i+1\in I_{i+1}(\tau)$. Thus, $I_{i+1}'(\tau)=I_{i+1}(\tau)\setminus\{i+1\}\bigcup\{i\}$ is well-defined.
    
Since $v_i\in \spann \{v_{i+1},v_{i+2},\cdots, v_{f(i)}\}$, we know that $v_i-\lambda v_{i+1}\in \spann \{v_{i+2},\cdots, v_{f(i)}\}$ for some $\lambda$. Furthermore, since $f(i+1)>f(i)$, $v_{i+1}\not\in \spann \{v_{i+2},\cdots, v_{f(i)}\}$. Thus, such a $\lambda$ is unique. And we show that $\lambda$ is uniquely determined by the formula $\frac{\Delta_{I_{i+1}'}(X)}{\Delta_{I_{i+1}}(X)}$.

Consider the equality $v_i=\lambda v_{i+1}+\sum_{j=i+2}^{f(i)}a_jv_j$ of column vectors. We can attach the same $k-1$ columns $\{v_k\mid k\in I_{i+1}\setminus\{i+1\}\}$ to vectors on both sides to form an equality of minors $\Delta_{I_{i+1}'}=\lambda \Delta_{I_{i+1}} +\sum_{j=i+2}^{f(i)}a_j\Delta_{I_{i+1}\setminus \{i+1\}\bigcup \{j\}}$.
    
Recall that $f(i+1)\not=i+1$. We can apply Lemma \ref{vanishing lem 1} to conclude that every term in the summation on the RHS is zero because $j\leq f(i)<f(i+1)$. We thus have $\lambda=\frac{\Delta_{I_{i+1}'}(X)}{\Delta_{I_{i+1}}(X)}$.
\end{proof}

\begin{lem}\label{vanishing lem 2}
    Let $\tau\lessdot\tau'$ be an $i$-crossing pair and $f,f'$ the corresponding bounded affine permutations. Then
    
    \begin{enumerate}
        \item $\Delta_{I_{i+1}(f')} $ vanishes on $\chi_f$.
        \item $\Delta_{I_{i+1}'(f')} $ does not vanish anywhere on $\mathring{\chi}_f$.
    \end{enumerate} 
\end{lem}
\begin{proof}

 For the first claim, by Lemma \ref{technical}, $I_{i+1}(f')=I_{i+1}(f)-\{f(i)\}\bigcup \{f(i+1)\}<_{i+1}I_{i+1}(f)$.
 
   For the second claim, by Remark \ref{r: no i,i-1}, $i,i-1\not\in I_{i+1}(f')$. Moreover, since $\tau'(i+1)\not=i+2$, it's guaranteed that $i+1\in I_{i+1}(f')$. Fix any $a\in \mathbb{C}^*$. By Lemma \ref{l:pluckerchange} case (iii), we have
 $\Delta_{I_{i+1}(f')}(u_i(a).X)=\Delta_{I_{i+1}(f')}(X)+a\Delta_{I'_{i+1}(f')}(X)$. By Proposition \ref{p:crossing}, $u_i(a).X\in \mathring{\chi}_{f'}$. So, the Grassmann necklace of $u_i(a).X$ is given by the Grassmann necklace $\cI(f')$. Thus, we have $0\not=\Delta_{I_{i+1}(f')}(u_i(a).X)=\Delta_{I_{i+1}(f')}(X)+a\Delta_{I'_{i+1}(f')}(X)$. By the first claim, $\Delta_{I_{i+1}(f')}$ vanishes on $\Xo_f$. So the previous inequality becomes $0\not=a\Delta_{I'_{i+1}(f')}(X)$ for $a\in \mathbb{C}^*$. This implies that $\Delta_{I_{i+1}'(f')}(X)$ must not vanish anywhere on $\mathring{\chi}_f$.
\end{proof}
We are now ready to prove the main result of this subsection.
\begin{proof}[Proof of Prop \ref{p:uncrossing adjacent}]
The target $\mathring{\chi}_{f'}$ has complement in $\chi_{f'}$ given by $\sqcup\Xo_g$ where the disjoint union is over all $g<f'$. Hence, the complement equals the union of closed electroid varieties ${\chi}_g$ for $g\lessdot f'$. We conclude that the target $\mathring{\chi}_{f'}$ is dense open in $\chi_{f'}$ and thus is a variety. 

By Proposition \ref{p:crossing}, we can define the morphism 
\begin{align*}
    \psi: \mathring{\chi}_{f}\times\mathbb{C}^* &\longrightarrow \mathring{\chi}_{f'}\\
    (Y,a) &\longmapsto u_i(a).Y
\end{align*}


Consider the open subscheme $U$ of $\mathring{\chi}_{f'}$ given by 
$$U:= \{X \in\mathring{\chi}_{f'}\mid \Delta_{I_{i+1}'(f')}(X)\not=0\}.$$ 
Notice that $b:=\frac{\Delta_{I_{i+1}(f')}(X)}{\Delta_{I'_{i+1}(f')}(X)}$ is a regular function on $U$. Furthermore, $\Delta_{I_{i+1}(f')}$ does not vanish anywhere on the open positroid variety $\mathring{\Pi}_{f'}$, and so it does not vanish anywhere on the open electroid variety $\mathring{\chi}_{f'}$. In particular, $b$ is a no-where vanishing regular function on $U$. By \cite[Proposition 5.25]{Lam2018}, $u_i(-b).X$ is in $\mathring{\chi}_{f}$. 

So we have a well-defined morphism $\phi:U\to \mathring{\chi}_{f}\times\mathbb{C}^*$ given by $$X\mapsto (u_i(-b).X, b).$$

We now show that $\phi$ and $\psi$ are two sided inverses of each other. First, the image of $\phi$ is contained in the domain of $\psi$ by definition. And for any $X\in U$, we have that $\psi\circ\phi(X)=\psi(u_i(-b).X,b)=u_i(b).u_i(-b).X=X$. We showed that $\psi\circ\phi=id$. 

Now we consider $\phi\circ\psi$. We must show that $\image\psi$ is contained in the domain $U$ of $\phi$. That is, we must show that $\Delta_{I_{i+1}'(f')}(u_i(a).X)\not=0$ for any $(X,a)\in \Xo_\tau\times \mathbb{C}^*$. By Remark \ref{r: no i,i-1}, we have that $i,i-1\not\in I_{i+1}(f')$ and $i+1\in I_{i+1}(f')$. Then, by Lemma \ref{l:pluckerchange}, we have
\begin{enumerate}
    \item $\Delta_{I_{i+1}(f')}(u_i(a).X)=\Delta_{I_{i+1}(f')}(X)+a\Delta_{I'_{i+1}(f')}(X)$
    \item $\Delta_{I_{i+1}'(f')}(u_i(a).X)=\Delta_{I_{i+1}'(f')}(X).$
\end{enumerate}
By Lemma \ref{vanishing lem 2}, we further have 
\begin{enumerate}
    \item $\Delta_{I_{i+1}(f')}(u_i(a).X)=a\Delta_{I'_{i+1}(f')}(X)$
    \item $\Delta_{I_{i+1}'(f')}(u_i(a).X)=\Delta_{I_{i+1}'(f')}(X)\not=0.$
\end{enumerate}
Thus, $\image\psi$ is contained in the domain $U$ of $\phi$. Now, $\phi\circ\psi(X,a)=\phi(u_i(a).X).$ We compute in the definition of $\phi$ that $b= \frac{\Delta_{I_{i+1}(f')}(u_i(a).X)}{\Delta_{I_{i+1}'(f')}(u_i(a).X)}=\frac{a\Delta_{I'_{i+1}(f')}(X)}{\Delta_{I_{i+1}'(f')}(X)}=a$. Since $b=a$, then $\phi$ maps $u_i(a).X$ to $(u_i(-a).u_i(a).X,a)=(X,a)$.

This completes the proof that $\Xo_{f}\times\mathbb{C}^* \longrightarrow \Xo_{f'}$ is an open embedding onto the open subscheme of $\mathring{\chi}_{f'}$ given by the nonvanishing of $\Delta_{I_{i+1}'(f')}$.
\end{proof}

\begin{eg}

Consider the $i$-crossing pair depicted in Figure \ref{f:embedding} $$\tau = (1,7),(2,6),(3,4),(5,8), \qquad \tau'=(1,6),(2,7),(3,4),(5,8).$$ Here $i=6$, $I_{i+1}(\tau)=712$, $I_{i+1}(\tau')=782$, and $I_{i+1}'(\tau')=682$. Since $u_6(a)$ acts on $X\in\Xo_\tau$ by adding $a$ times column 6 to columns 5 and 7, 
$$\Delta_{782}(u_6(a).X) = \Delta_{782}(X) + a\Delta_{682}(X).$$ Since $782\in\cI_{\tau'}$ and $u_6(a).X\in\Xo_{\tau'}$, the LHS is nonzero. We have that $782 <_7 712=I_{7}(\tau)$ and $X\in \Xo_\tau$, so $\Delta_{782}(X)=0$. Thus, $\Delta_{682}(X)$ is nonzero. Since the action of $u_6(a)$ doesn't change columns $6,8,$ or $2$, we have $\Delta_{682}(u_i(a).X)=\Delta_{682}(X)\not=0$ as desired.
\label{e:embedding}
\end{eg}

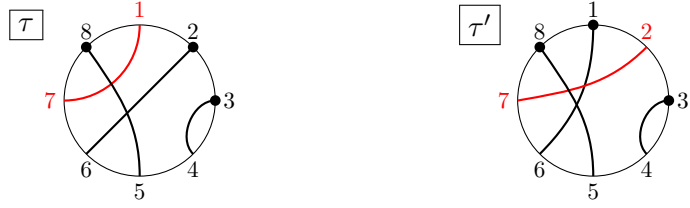
\begin{figure}
    \centering
    \begin{tikzpicture}
        \coordinate (A) at (-3,0);
        \coordinate (B) at (3,0);
        
        \draw (A) circle (1);
        \draw (B) circle (1);

        \draw [shift={(-3,0)},thick, red] (0,1) to[out=270,in=0] (-1,0);
        \draw [shift={(-3,0)},thick] (.707,.707) to[out=225,in=45] (-.707,-.707);
        \draw [shift={(-3,0)},thick] (1,0) to[out=180,in=135] (.707,-.707);
        \draw [shift={(-3,0)},thick] (0,-1) to[out=90,in=305] (-.707,.707);

         \draw (A) ++(0,1) node [above,scale=.8,red] {1};
         \draw (A) ++(1,0) node [right,scale=.8] {3};
         \draw (A) ++(0,-1) node [below,scale=.8] {5};
         \draw (A) ++(-1,0) node [left,scale=.8,red] {7};
         \draw (A) ++(.707,.707) node [above,scale=.8] {2};
         \draw (A) ++(.707,-.707) node [below,scale=.8] {4};
         \draw (A) ++(-.707,-.707) node [below,scale=.8] {6};
         \draw (A) ++(-.707,.707) node [above,scale=.8] {8};

         \fill (A) ++(-.707,.707) circle (2pt);
         \fill (A) ++(1,0) circle (2pt);
         \fill (A) ++(.707,.707) circle (2pt);

         \draw [shift={(3,0)},thick] (0,1) to[out=270,in=45] (-.707,-.707);
        \draw [shift={(3,0)},thick,red] (.707,.707) to[out=225,in=10] (-1,0);
        \draw [shift={(3,0)},thick] (1,0) to[out=180,in=135] (.707,-.707);
        \draw [shift={(3,0)},thick] (0,-1) to[out=90,in=305] (-.707,.707);

         \draw (B) ++(0,1) node [above,scale=.8] {1};
         \draw (B) ++(1,0) node [right,scale=.8] {3};
         \draw (B) ++(0,-1) node [below,scale=.8] {5};
         \draw (B) ++(-1,0) node [left,scale=.8,red] {7};
         \draw (B) ++(.707,.707) node [above,scale=.8,red] {2};
         \draw (B) ++(.707,-.707) node [below,scale=.8] {4};
         \draw (B) ++(-.707,-.707) node [below,scale=.8] {6};
         \draw (B) ++(-.707,.707) node [above,scale=.8] {8};

         \fill (B) ++(-.707,.707) circle (2pt);
         \fill (B) ++(1,0) circle (2pt);
         \fill (B) ++(0,1) circle (2pt);

         \draw (-4.5,1) node {\fbox{$\tau$}};
         \draw (1.5,1) node {\fbox{$\tau'$}};

    \end{tikzpicture}
    \caption{The pairings $\tau = (1,7),(2,6),(3,4),(5,8)$ and $\tau'=(1,6),(2,7),(3,4),(5,8)$ with subsets $I_{7}(\tau)+1=823$ and $I_{7}(\tau')+1=813$ indicated by the dots.}
    \label{f:embedding}
\end{figure}

\subsection{Proof of Irreducibility}

\begin{proof}[Proof of Theorem \ref{Main1}.1 (Irreducible, Smooth, Dimension Count).]

We proceed by induction on the codimension $d:=\binom{n}{2} - c(\tau)$. When $d=0$, the only pairing is $\tau_\text{top} = (1,1+n), (2,2+n), \dots ,(n,2n)$. Per \cite{CGS2021}, $\chi$ is abstractly isomorphic to the Lagrangian Grassmannian. As the standard Lagrangian Grassmannian is a homogeneous space for the connected algebraic group $Sp(2n,\mathbb{C})$, it irreducible, and so is $\chi$. Since $\Xo_{\tau_\text{top}}$ is a Zariski open subset of $\chi$, the variety $\Xo_{\tau_\text{top}}$ is also irreducible.

For $\Xo_\tau$ with $d>0$, there exists some strands $i$ and $i+1$ which do not cross in $\tau$. To see this, take noncrossing strands $i<j$ with $j-i$ minimal. Any strand $k$ with $i<k<j$ cannot cross both $i$ and $j$, thereby contradicting the minimality of $j-i$. Thus, there exists $\tau'$ such that $\tau\lessdot\tau'$ is an $i$-crossing pair and $c(\tau')=c(\tau)+1$. By Proposition \ref{p:uncrossing adjacent}, $\psi:\Xo_{\tau}\times\mathbb{C}^* \hookrightarrow \Xo_{\tau'}$ is an open embedding onto an open subscheme of $\mathring{\chi}_{\tau'}$. By the inductive hypothesis, $\Xo_{\tau'}$ is irreducible and smooth of dimension $c(\tau')$. Thus, the isomorphic image of $\Xo_{\tau}\times\mathbb{C}^*$ must be dense in $\Xo_{\tau'}$, and so it is irreducible of the same dimension. This implies that $\Xo_\tau$ is irreducible and smooth of dimension $c(\tau)$. 
\end{proof}


\section{Frobenius Splitting and its Consequences}

\subsection{A quick overview of Frobenius splitting}
We provide a quick overview of the theory of Frobenius splitting. 
The classic text here is~\cite{BrionKumar}; our presentation is also closely based on~\cite{Knutson2009}.
Let $k$ be a field of characteristic $p$, and let $X$ be a scheme over $k$. 
\begin{defn}
A \emph{Frobenius splitting} on $X$ is a map of sheaves $\phi : \cO_X \to \cO_X$ obeying the relations:
\[ \phi(u+v) = \phi(u)+\phi(v), \quad \phi(a^p b) = a \phi(b),\ \quad \phi(1) = 1.\]
If $\phi$ only obeys the first two relations, it is called a near splitting.    
\end{defn}
Note that the set of all near-splittings is a $k$-vector space, and the set of splittings is an affine linear $\mathbb{F}_p$-subspace of that vector space.

Most schemes do not support Frobenius splittings.
If $X$ supports a Frobenius splitting, then $X$ is necessarily reduced~\cite[Proposition 1.2.1]{BrionKumar} and weakly normal (also called semi-normal)~\cite[Proposition~1.2.5]{BrionKumar}. 

\begin{defn}
Given a scheme $X$ with a Frobenius splitting $\phi$, an ideal sheaf $\cI \subseteq \cO_X$ is called \emph{compatibly split} if $\phi(\cI) \subseteq \cI$. 
\end{defn}
If $\cI$ is compatibly split, then $\phi$ descends to an splitting on the closed subscheme $V(\cI)$ corresponding to $\cI$, and so $V(\cI)$ is reduced, and is therefore defined by its support. 
Thus, given a closed subvariety (not necessarily irreducible) $Y$ of $X$, we say that $Y$ is compatibly split if its radical ideal sheaf is compatibly split.
There are a number of operations which let us produce new compatibly split subschemes from old ones.

\begin{prop}\cite[Proposition 1.2.1]{BrionKumar} \label{IntersectDecompose}
    Let $X$ be a scheme with a Frobenius splitting $\phi$. Let $Y$ and $Z$ be compatibly split subvarieties of $X$. 
    Then $Y \cup Z$, and the scheme theoretic intersection $Y \cap Z$, are compatibly split; in particular, the scheme theoretic intersection $Y \cap Z$ is reduced.
    Also, if $Y \subset X$ is compatibly split, and $W$ is an irreducible component of $Y$, then $W$ is compatibly split.
\end{prop}

This raises the issue of how to find splittings in the first place. 
The following result is from~\cite[Section 1.3]{BrionKumar}:
\begin{thm}
    Let $X$ be smooth over $k$ and let $\omega_X$ be the canonical bundle. Then the vector space of near splittings on $X$ is naturally identified with $H^0(X, \omega_X^{1-p})$.
\end{thm}
In particular, given an anti-canonical section $\sigma \in H^0(X, \omega_X^{-1})$, the power $\sigma^{p-1}$ defines a near-splitting $\phi$ on $X$. 
We want to know whether this near splitting is a splitting, in other words, whether or not $\phi(1)=1$. 
Since $\phi$ is a map of sheaves, it will be enough to verify that $\phi(1)=1$ on any dense open subset of $X$.
In our application, we will in particular work on in a dense open affine which is isomorphic to $\mathbb{A}^n$.

\begin{thm} \cite[Theorem 2]{Knutson2009} \label{LeadingTermCriterion}
  Let $f$ be a polynomial of degree $n$ in $k[x_1, x_2, \ldots, x_n]$. Suppose that we have a term order $\prec$ on $k[x_1, x_2, \ldots, x_n]$ such that the leading term of $f$ is a nonzero $\mathbb{F}_p$ multiple of $x_1 x_2 \cdots x_n$. Then, $(f (dx_1 \wedge dx_2 \wedge \cdots \wedge dx_n)^{-1})^{p-1}$ induces a splitting on $\text{Spec}\ k[x_1, x_2, \ldots, x_n]$. Moreover, the zero locus of $f$ is compatibly split for this splitting.
\end{thm}
So, if $\sigma \in H^0(X, \omega_X^{-1})$, and if we can find a dense open affine chart $U$ in $X$ isomorphic to $\text{Spec}\ k[x_1, x_2, \ldots, x_n]$ such that $\sigma|_U$ is of the form $f (dx_1 \wedge dx_2 \wedge \cdots \wedge dx_n)^{-1}$ for $f$ as in Theorem~\ref{LeadingTermCriterion}, then $\sigma^{p-1}$ induces a splitting on $X$.

In particular, if $X$ is projective and smooth, then we can specify a section of $\omega_X^{-1}$ by specifying the anticanonical divisor on which $f$ vanishes. Thus, our recipe for building Frobenius splittings in progress is to
\begin{enumerate}
    \item Specify an anticanonical section $D$ in $X$.
    \item Find dense open affine subset of $X$, isomorphic to $\AA^n$, where $D$ is given by the vanishing of a polynomial $f$ as in Theorem~\ref{LeadingTermCriterion}.
\end{enumerate}
We then get a splitting, where the divisor $D$ is compatibly split, and we can try to use Proposition~\ref{IntersectDecompose} to find more compatibly split subvarieties. We now carry out this program for the electroid space.

In this section, it is convenient to embed the electroid space into the Grassmannian $\Gr(n+1, 2n)$.


\subsection{A Chart Computation}

\begin{thm}[\protect{\cite[Theorem 1.6]{CGS2021}}]
The electroid space $\chi_n\subset \Gr(n+1,2n)$ contains two affine opens $\nU$ and $\cU$, both isomorphic to $\bA^{\binom{n}{2}}$. 

Points on $\nU\subset \Gr(n+1,2n)$ are given by $[XD]$, where
\begin{equation}
    X=\begin{bmatrix}
    0&1 &0&1&\cdots&0&1\\
    1&S_{11}&0&S_{12}&\cdots&0&S_{1n}\\
    0&S_{21}&1&S_{22}&\cdots&0&S_{2n}\\
    \vdots&\vdots&\vdots&\vdots&\cdots&\vdots&\vdots\\
    0&S_{n1}&0&S_{n2}&\cdots&1&S_{nn}\\
\end{bmatrix}
\end{equation}  
The isomorphism $\nU\simeq \bA^{\binom{n}{2}}$ is given explicitly by taking $L_{ji}=S_{i,{j-1}}-S_{ij}$, where the $L$'s form the response matrix of the electric network.

Points on $\cU\subset \Gr(n+1,2n)$ are given by $[\Tilde{X}D]$, where
\begin{equation}
    \Tilde{X}=\begin{bmatrix}
    1 &0&1&0&\cdots&1&0\\
    T_{11}&1&T_{12}&0&\cdots&T_{1n}&0\\
    T_{21}&0&T_{22}&1&\cdots&T_{2n}&0\\
    \vdots&\vdots&\vdots&\vdots&\cdots&\vdots&\vdots\\
    T_{n1}&0&T_{n2}&0&\cdots&T_{nn}&1\\
\end{bmatrix}
\end{equation}  
The isomorphism $\cU\simeq \bA^{\binom{n}{2}}$ is given explicitly by taking $L_{ji}^*=T_{i,{j+1}}-T_{ij}$, where the $L$'s form the response matrix of the dual of the cactus network.
\end{thm}

Let $\mathcal{C}$ be the set of electrical affine permutations $f$ arising from a medial pairing with exactly $\binom{n}{2}-1$ crossings. For each $f\in \mathcal{C}$, $\Pi_f$ is a positroid variety of codimension $2$ in $\Gr(n+1,2n)$ given by the vanishing of two Pl\"{u}cker coordinates indexed by the cyclic intervals $[i,i+1, \dots ,i+n]$ and $[i+n, i+n+1, \dots, i+2n]$ of size $n+1$. 

\begin{lem}
    On $\Xn$, we have the equality of Pl\"ucker coordinates \[\Delta_{[i,i+1, \dots ,i+n]} = \Delta_{[i+n, i+n+1, \dots, i+2n]}.\]
\end{lem}

\begin{proof}
    Each of the sets $[i,i+1, \dots ,i+n]$ and $[i+n, i+n+1, \dots, i+2n]$  is concordant with only one set partition, namely $\{ \{i\} , \{i-1, i+1 \}, \{ i-2, i+2 \} ,\ldots, \{i+n\} \}$. So we have
    \[ \Delta_{[i,i+1, \dots ,i+n]} = L_{\{i\} , \{i-1, i+1 \}, \{ i-2, i+2 \} ,\ldots, \{i+n\}} =  \Delta_{[i+n, i+n+1, \dots, i+2n]}. \qedhere \]
\end{proof}
 So, for each $f \in \mathcal{C}$, the electroid variety $\chi_f$ is cut out by the vanishing of one interval Pl\"{u}cker inside $\Xn$ and is a divisor on $\chi_n$. 
 That is, $\chi_f$ is the zero locus of a global section of $\mathcal{O}_{\Xn}(1)$. Since there are $n$ such pairs of Pl\"{u}cker coordinates, there are $n$ distinct electroid divisors. 
 We will eventually show that the union of the $n$ electroid divisors is anticanonical, and induces a Frobenius splitting.


We now introduce local coordinates in the affine chart of unshorted cactus networks.
Recall that the response matrix $L=\{L_{ij}\}_{i,j=1}^{n}$ of an unshorted cactus network is a symmetric matrix whose rows and columns sum to zero. So the coordinate ring of the space of response matrices can be identified with the polynomial ring $k[L_{ij}:i<j]$. In this section, we will work with a monomial term order $\prec$ on $k[L_{ij}:i<j]$; see~\cite{MS2004} for the basic theory of term orders.

We chose our term order $\prec$ such that, for any $1 \leq p < q < r < s \leq n$, we have 
$L_{pq}L_{rs} \succ L_{pr}L_{qs}$ and $L_{ps}L_{qr} \succ L_{pr}L_{qs}$. (In other words, the two non-crossing monomials both dominate the monomial with a crossing.)
A specific term order which works is to order the variables in any manner such that $L_{ab} \succ L_{cd}$ if $|a-b|<|c-d|$, and then use the corresponding lexicographic order on monomials.


Let $L_I^{J}$ denote the minor of $L$ corresponding to rows in $I$ and columns in $J$. 

\begin{prop}\label{prop:sqrfree on not shorted}
   Let $n\geq3$. Fix coordinates $L_{ij}$ on $ \nU\simeq \operatorname{Spec} k[L_{ij}]$. Each electroid divisor inside $\nU$ is given by the vanishing of a polynomial in $k[L_{ij}]$. The initial term of the product of the $n$ electroid divisors is $\Pi_{i<j} L_{ij}$ (under the prescribed term order $\prec$).
\end{prop}

\begin{proof}
Keep in mind that we index the columns by $1,\Tilde{1},2,\Tilde{2},\cdots ,n,\Tilde{n}$ and we index the rows by $0,1,\cdots n$. Thus, $S_{ij}$ is positioned at row $i$ and column $\Tilde{j}$. Every odd column $j$ has exactly a $1$ at row $j$ and all other entries being $0$.

For $n$ odd and even, we respectively compute all the Pl\"uckers indexed by an interval of length $n+1$ as a polynomial in $k[L_{ij}]$. Multiplying by $D$ on the right only changes the sign of a maximal minor, so we may only compute the maximal minors of $X$ indexed by an interval of length $n+1$.

\textbf{Case 1:} $n=2k+1$.

We can represent all minors as $X_{\Tilde{i},i+1, \cdots,\widetilde{i+k},{i+k+1}}.$ For $j=i+1,\cdots,i+k+1$, every odd column $j$ has a $1$ in row $j$ and all other entries $0$. Therefore, we can column expand on the odd columns, removing column $j$ and row $j$. The determinant is, up to sign, equal to the determinant of the $k+1$ by $k+1$ minor of the matrix
$$X'=\begin{bmatrix}
    1&1 &\cdots&1\\
    S_{1,{i}}&S_{1,{i+1}}&\cdots&S_{1,{i+k}}\\
    \vdots&\vdots&\cdots&\vdots\\
    S_{2k+1,{i}}&S_{2k+1,{i+1}}&\cdots&S_{2k+1,{i+k}}\\
\end{bmatrix}$$ 
where we omit the $S$'s with its first index in $j=i+1,\cdots,i+k+1$. To compute its determinant, we do the following column operations, starting with the last column and ending with the first column. For each column, subtract the preceding column from it. We get a matrix
$$X''=\begin{bmatrix}
    1&0 &\cdots&0\\
    S_{1,{i}}&S_{1,{i+1}}-S_{1,{i}}&\cdots&S_{1,{i+k}}-S_{1,{i+k-1}}\\
    \vdots&\vdots&\cdots&\vdots\\
    S_{2k+1,{i}}&S_{2k+1,{i+1}}-S_{2k+1,{i}}&\cdots&S_{2k+1,{i+k}}-S_{2k+1,{i+k-1}}\\
\end{bmatrix}.$$
And this equals 
$$\begin{bmatrix}
    1&0 &\cdots&0\\
    S_{1,{i}}&-L_{{i+1},1}&\cdots&-L_{{i+k},1}\\
    \vdots&\vdots&\cdots&\vdots\\
    S_{2k+1,{i}}&-L_{{i+1},2k+1}&\cdots&-L_{{i+k},2k+1}\\
\end{bmatrix}$$  
where we omit the $S$'s with its first index in $j=i+1,\cdots,i+k+1$. Finally, we may expand on the first row, which is the same as deleting the first column. We are left with, up to sign, the minor of the matrix $(L_{pq})$ with $p= i+1,\cdots, i+k$ and $q={i+k+2,\cdots,i}$ 

We conclude that up to sign, $\Delta_{i,i+1,\cdots,i+n}$ equals $L_{i+1,\cdots,i+k}^{i+k+2,\cdots, i}$. In other words, we get all $L_I^J$ for $I$ and $J$ disjoint intervals with $k$ elements each. There are $n$ of these $L^J_I$'s by varying $i$. Each monomial in the expansion of this determinant is given by a pairing of elements in the two cyclic intervals $[i+k+2,i]$ and $[i+1,i+k]$ of length $k$. By the term order that we picked, if we arrange $1$ through $n$ on a circle, any crossing can be uncrossed and the resulting monomial will go up in the term order. So the initial term is given by the only pairing with no crossings. That is, $L_{i,i+1} L_{i-1,i+2},\cdots,L_{i+k+2,i+1}$. Letting $i$ range through $[n]$ hits every $L_{ij}$ with $i<j$ exactly once.

\textbf{Case 2:} n=2k. We have two families of minors given in the following discussion.

\textbf{Case 2.1:} Consider minors given by $X_{i,\Tilde{i},\cdots, \widetilde{i+k-1},i+k}$ for $ i\in[k]$. For $j=i,\cdots,i+k$, every odd column $j$ has a $1$ in row $j$ and all other entries $0$. As before, we can column expand on the odd columns, removing columns $j$ and rows $j$. The determinant is, up to sign, equal to the determinant of the $k$ by $k$ minor of the matrix $$X'=\begin{bmatrix}
    1&1 &\cdots&1\\
    S_{1,{i}}&S_{1,{i+1}}&\cdots&S_{1,{i+k-1}}\\
    \vdots&\vdots&\cdots&\vdots\\
    S_{2k+1,{i}}&S_{2k+1,{i+1}}&\cdots&S_{2k+1,{i+k-1}}\\
\end{bmatrix}$$
where we omit the $S$'s with its first index in $j=i,\cdots,i+k$
The same reasoning as in the previous case shows that up to sign, $\Delta_{i,\Tilde{i},\cdots,\widetilde{i+k-1},i+k}$ equals $L_{i+1,\cdots,i+k-1}^{i+k+1,\cdots i-1}$

\textbf{Case 2.2:} Consider minors given by $X_{\Tilde{i},i+1,\cdots ,i+k,\widetilde{i+k}} , \quad i\in [k].$ For $j=i+1,\cdots,i+k$, we may remove columns $j$ and rows $j$. By the same reasoning, we conclude that up to sign, $X_{\Tilde{i},i+1,\cdots,i+k,\widetilde{i+k}}=L_{i,\cdots,i+k}^{i+k+1,\cdots, i}$

By the same reasoning as Case 1, the initial term is given by pairing the upper and lower indexing sets such that there are no crossings when embedded in a disk. And it is straightforward to check that combining the two families, we also run through all $L_{ij}$ with $i<j$ exactly once.
\end{proof}

\begin{eg}
When $n=5$, the five electroid divisors are given by the vanishing of Pl\"ucker coordinates $\Delta_{123456}, \Delta_{234567}, \dots, \Delta_{56789,10,11}$. All five divisors lie in the local chart $\nU$ of $\IG^{\Omega}(n+1, 2n)$. On $\nU$, the Pl\"ucker coordinate $\Delta_{i,i+1,\dots, i+5}$ is the minor of the response matrix $L$
in rows $i+1$, $i+2$ and columns $i+4$, $i+5 \bmod 5$, with leading term $L_{(i+1) (i+4)} L_{(i+2)(i+3)}$. For example, $\Delta_{123456} = L_{24} L_{35} - L_{25} L_{34}$ has leading term $L_{25} L_{34}$. 
Thus, the local expansion for these five electroid divisors has leading term $\prod_{i=1}^5 {\big(} L_{(i+1) (i+4)} L_{(i+2)(i+3)} {\big)} = \prod_{1 \leq i < j \leq 5} L_{ij}.$ 
\label{e:charts}
\end{eg}

\subsection{Proof of Frobenius Splitting}

We start by presenting two lemmas. The first, Lemma \ref{omega of LG}, will be used to show that electroid hypersurfaces are Frobenius split. The second, Proposition \ref{2 coverings}, will be used to extend the result to all electroid varieties.

\begin{lem}\label{omega of LG}
The Lagrangian Grassmannian $\LG(n-1,2n-2)$ and the electroid space $\chi_n$ have anticanonical bundle given by $\mathcal{O}(n)$ (using the $\mathcal{O}(1)$ from the Pl\"ucker embedding).
\end{lem}
\begin{proof}
    We first show that the Lagrangian Grassmannian $\LG(n,2n)$ has anticanonical bundle $\mathcal{O}(n+1)$. Let $V=\mathbb{C}^{2n}$ and $\Gr(n,2n)=\Gr(n,V)$. It is well known that the Grassmannian $\Gr(n,V)$ has canonical bundle $\mathcal{O}(-2n)$. Let $\mathcal{I}$ be the ideal sheaf of the closed embedding $\iota:\LG(n,V)\hookrightarrow \Gr(n,V)$. By the adjunction formula, we have the canonical bundle on $\LG$ 
    \[\omega_{\LG}=\iota^* \omega_{\Gr}\otimes \det(\mathcal{I}/\mathcal{I}^2)^\vee \]
    Let $S$ be the tautological vector bundle on $\Gr(n,2n)$. Its dual $S^\vee$ is a rank $n$ vector bundle whose fiber over $W\in \Gr(n,V)$ is $W^\vee$. A nondegenerate antisymmetric bilinear form $\Omega$ on $V$ defines a global section of $\bigwedge^2 S^\vee$; over each $W\in \Gr(n,V)$, $\Omega\in \bigwedge^2V^\vee$ restricted to $W$ produces $\Omega_{|_W}\in\bigwedge^2 W^\vee$, which is the fiber of $\bigwedge^2 S^\vee$ at $W$. Thus, $\LG=\{W\in \Gr(n,V): \Omega_{|_W}=0\}$ is the zero locus in $\Gr(n,V)$ of a section of $\bigwedge^2 S^\vee$. Thus, the normal bundle $(\mathcal{I}/\mathcal{I}^2)^\vee$ of $\iota:\LG(n,V)\hookrightarrow \Gr(n,V)$ is given by $\iota^*\bigwedge^2 S^\vee$

    Now, we compute the first Chern class of the normal bundle.
    \begin{align*}
        c_1(\det (\mathcal{I}/\mathcal{I}^2)^\vee)
        &=c_1(\det (\iota^*\bigwedge^2 S^\vee))\\&=c_1(\iota^*\bigwedge^2 S^\vee) \text{  by the Splitting principle }\\
        &=\iota^*c_1(\bigwedge^2 S^\vee) \text{  Chern class commutes with flat pullback}\\
        &=\iota^*((n-1)c_1(S^\vee)) \text{  by the Splitting principle}\\
        &=\iota^*((n-1)c_1(\mathcal{O}(1))) \text{  by \cite[Page 178]{EH2016}}\\
        &=c_1(i^*\mathcal{O}(n-1))=c_1(\mathcal{O}(n-1)).        
    \end{align*}
Since line bundles are determined up to isomorphism by their first Chern classes $L=O_X(c_1(L))$, we conclude that $\det (\mathcal{I}/\mathcal{I}^2)^\vee)=\mathcal{O}(n-1)$. Thus, $\omega_{\LG}=\iota^* \omega_{\Gr}\otimes \det(\mathcal{I}/\mathcal{I}^2)^\vee=\mathcal{O}(-2n)\otimes \mathcal{O}(n-1)=\mathcal{O}(-n-1)$. 

Now, we consider the electroid space. Fix a $2n$-dimensional vector space $V$. Recall that $\Omega$ is now an anti-symmetric form on $V$ with $2$-dimensional $\ker \Omega$. We reduce to the previous case as follows. Notice that the embedding of $\Gr(n-1, V/\ker\Omega)\hookrightarrow\Gr(n+1,V)$ as the closed Schubert variety consisting of subspaces containing $\ker\Omega$ lifts to an embedding of the respective Pl\"ucker spaces as a linear subspace. More explicitly, consider Pl\"ucker embeddings $$\Gr(n-1, V/\ker\Omega)\hookrightarrow\mathbb{P}(\bigwedge^{n-1 } (V/\ker\Omega))=\mathbb{P}^{\binom{2n-2}{n-1}-1}, \qquad  \spann\{v_1,\cdots,v_{n-1}\}\mapsto v_1\wedge\cdots \wedge v_{n-1}$$ and 
$$\Gr(n+1, V)\hookrightarrow\mathbb{P}(\bigwedge^{n+1 } V)=\mathbb{P}^{\binom{2n}{n+1}-1}, \qquad  \spann\{v_1,\cdots,v_{n+1}\}\mapsto v_1\wedge\cdots \wedge v_{n+1}.$$ 
The lift $\mathbb{P}(\bigwedge^{n-1 } (V/\ker\Omega))\hookrightarrow \mathbb{P}(\bigwedge^{n+1 } V)$ is given by $\tilde{v}_1\wedge \cdots \wedge \tilde{v}_{n-1} \mapsto v_1\wedge\cdots \wedge v_{n-1} \wedge a_1 \wedge a_2$,
where $a_1$ and $a_2$ are any fixed basis vectors for the $2$-dimensional $\ker V$ and $v_i$ is any lift of $\tilde{v}_i\in V/\ker \Omega$ to $V$. Thus, $\mathcal{O}(1)$ pulls back to $\mathcal{O}(1)$ via $\Gr(n-1, V/\ker\Omega)\hookrightarrow\Gr(n+1,V)$. Thus, the same holds for $\IG^{\Omega}(n+1,V)\simeq \IG^{\Omega}(n-1,V/\ker\Omega)$ after passing to isotropics. Finally, the identification $\IG^{\Omega}(n-1,V/\ker\Omega)$ with $\LG(n-1, V/\ker\Omega)$ is by a change of nondegenerate antisymmetric form, which is a linear change of coordinates on $V/\ker\Omega$. So this also lifts to a linear change of coordinate on Pl\"ucker spaces. Thus, the identification $\LG(n-1,V/\ker \Omega)\simeq \IG^{\Omega}(n+1,V)$ identifies $\mathcal{O}(1)$ on both sides. It follows from the previous part that $\chi_n$ also has anticanonical bundle $\mathcal{O}(n)$.
\end{proof}

\begin{prop}\label{2 coverings}
In the poset of pairings on $[2n]$, every interval of length $2$ is a diamond. Namely, for any $f<g$ such that $
l(f)+2=l(g)$, we have exactly two elements $h_1,h_2$ such that $f<h_1,h_2<g$. Furthermore, $h_1$ and $h_2$ are not comparable.
\end{prop}
\begin{proof}
The poset of pairings is Eulerian \cite{L15}. The Mobius function characterization of Eulerian posets immediately implies the proposition.   
\end{proof}

\begin{proof}[Proof of Theorem \ref{Main1}.4 (Frobenius Splitting).]
The strategy of this proof is as follows: First, we show that the $n$ electroid divisors are compatibly split. Then, by induction on codimension, we show that all other electroid varieties are split.

By Lemma~\ref{omega of LG}, the union of the $n$ electroid divisors in anti-canonical, so let $\sigma$ be a section of $\omega^{-1}$ vanishing (to order $1$) on the electroid divisors. 
Using Lemma~\ref{prop:sqrfree on not shorted}, the criterion of Theorem~\ref{LeadingTermCriterion} applies to $\sigma^{p-1}$, so $\sigma^{p-1}$ induces a splitting. 
By Theorem~\ref{LeadingTermCriterion}, the zero locus of $\sigma$, which is the union of the electroid divisors, is compatibly split. 
By Theorem~\ref{IntersectDecompose}, each irreducible component of this zero locus, which is to say, each electroid divisor, is compatibly split. 

We will now prove, by induction on $c$, that every electroid variety of codimension $c$ is compatibly split. 
We have done the base case, $c=1$, above.
Let $\chi_{\tau}$ be an electroid variety of codimension $c>1$ and suppose that all electroid varieties of lower codimension are compatibly split. 
Since $\El(n)$ is graded with a unique maximal element, we can find $\tau_2 \gtrdot \tau_1 \gtrdot \tau$. And then, since the interval $[\tau, \tau_2]$ is a diamond (Lemma~\ref{2 coverings}), we can find a second $\tau'_1 \neq \tau_1$ with $\tau_2 \gtrdot \tau'_1 \gtrdot \tau$.

By induction, $X_{\tau_1}$ and $X_{\tau'_1}$ are compatibly split, so $X_{\tau_1} \cap X_{\tau'_1}$ is compatibly split. In particular, $X_{\tau_1} \cap X_{\tau'_1}$ is reduced. Since $X_{\tau}$ is irreducible and codimension $1$ inside each of $X_{\tau_1}$ and $X_{\tau'_1}$, we see that $X_{\tau}$ is an irreducible component of $X_{\tau_1} \cap X_{\tau'_1}$, so $X_{\tau}$ is compatibly split. 
\end{proof}

\subsection{Consequences of Frobenius Splitting}

\begin{proof}[Proof of Theorem \ref{Main1}.2, \ref{Main1}.5, and \ref{Main1}.3 except R1]
Reducedness is guaranteed by being compatibly split \cite{BrionKumar}. The fact that closed electroid varieties are the disjoint union of open electroid varieties is proved in \cite{Lam2018}, also a consequence of Theorem \ref{cell decomp}. Thus, $\chi_\tau=\bigsqcup_{\mu\leq\tau} \Xo_\mu$ is a Zariski closed set containing $(\Xo_{\tau})_{\geq0}$. We induct on $c(\tau)$ to show that the Zariski closure of $(\Xo_{\tau})_{\geq0}$ contains $\chi_\tau$.

When $c(\tau)=0$, $(\Xo_{\tau})_{\geq0}$ is a point. By irreducibility and reducedness of the open electroid varieties and the dimension count, $\Xo_\tau=\chi_\tau$ is a reduced point.

Suppose the result holds for all $c(\tau)<c$. Take a pairing $\tau$ with $c(\tau)=c$. Recall from Theorem \ref{ElectricalSummary} that the analytic closure of $(\Xo_{\tau})_{\geq0}$ equals $\bigsqcup_{\mu\leq\tau}(\Xo_{\mu})_{\geq0}$. Since the Zariski topology is coarser than the analytic topology, the Zariski closure of $(\Xo_{\tau})_{\geq0}$ contains the Zariski closure of $\bigsqcup_{\mu\leq \tau}(\Xo_{\mu})_{\geq0}$. In particular, this contains the Zariski closure of $(\Xo_{\mu})_{\geq0}$ for each $\mu<\tau$, which equals $\chi_\mu$ by the inductive hypothesis. On the other hand, we know from Theorem \ref{Main1}.1 that $\Xo_\tau$ is irreducible with Krull dimension $c(\tau)$. By Theorem \ref{ElectricalSummary}, $(\Xo_{\tau})_{\geq0}$ is a ball of the same dimension, $c(\tau)$, contained in $\Xo_\tau$. So, $(\Xo_{\tau})_{\geq0}$ is dense in $\Xo_\tau$. 
    
Combining the above, the Zariski closure of $(\Xo_{\tau})_{\geq0}$ contains $\Xo_\tau \bigcup (\bigcup_{\mu<\tau}\chi_\mu)=\chi_\tau$ as a set. Since the right hand side is reduced, we have containment as a scheme. Thus, the electroid cell is dense in the open electroid variety, which in turn is dense in the closed electroid variety. Irreducibility and dimension count for closed electroid varieties follow from those for electroid cells and open electroid varieties.
\end{proof}

With these results, we can prove a stronger version of Proposition \ref{p:uncrossing adjacent}, which will be of use in the next section.

\begin{prop}\label{p:uncrossing adjacent strong}
We have an open embedding of algebraic varieties
    \begin{align*}
    \psi: \mathring{\chi}_{\tau}\times\mathbb{C} &\longrightarrow \mathring{\chi}_{\tau}\bigsqcup \mathring{\chi}_{\tau'}\\
    (Y,a) &\longmapsto u_i(a).Y
\end{align*}
where the target is an open subscheme of $\chi_{\tau'}$. The image of $\psi$ contains $\mathring{\chi}_{\tau}$ and is the dense open subscheme of the target given by the nonvanishing of $\Delta_{I_{i+1}'(\tau')}$. The inverse of $\psi$ can be obtained by $X\mapsto (u_i(-a).X,a)$, where $a=\frac{\Delta_{I_{i+1}(\tau')}}{\Delta_{I_{i+1}'(\tau')}}$.
\end{prop}

\begin{proof}
We may a priori define the extended morphism $\psi:(Y,a) \longmapsto u_i(a).Y$ from $\mathring{\chi}_{f}\times\mathbb{C}$ to the Grassmannian. Since $\psi(Y,0)=u_i(0).Y=Y$, the set theoretic image of $\psi$ is $\{X\in \mathring{\chi}_{f'} \mid \Delta_{I_{i+1}'(f')}(X)\not=0 \}\bigsqcup \mathring{\chi}_{f}$. By Theorem \ref{Main1}.5, the Zariski closure of the image of $\psi$ is the closed electroid variety $\chi_{f'}$. The source is reduced, so the scheme theoretic image of $\psi$ is $\chi_{f'}$. Thus, the extended $\psi$ factors through the closed embedding of the closed electroid into the Grassmannian.

We showed in Lemma \ref{vanishing lem 2}(ii) that $\Delta_{I_{i+1}'(f')}$ does not vanish on the entire $\mathring{\chi}_{f}$. So the set theoretic image of $\psi$ is indeed equal to the dense open subscheme $\Xo_f\bigsqcup \Xo_{f'}$ in $\chi_{f'}$. Finally, the inverse is inherited from Proposition \ref{p:uncrossing adjacent}.
\end{proof}

\section{Regularity in Codimension One}

Given a covering relation $\tau\lessdot\tau'$, call $\tau\lessdot\tau'$ a \emph{regular pair} if the closed electroid variety $\chi_{\tau'}$ is regular along the generic point of the open electroid variety $\Xo_{\tau}$. From Theorem \ref{Main1}.1, open electroid varieties are smooth. Thus, regularity in codimension $1$ of closed electroid varieties is guaranteed by asserting that all covering relations are regular pairs. 

Since open electroids are smooth, Proposition \ref{p:uncrossing adjacent strong} shows that $i$-crossing pairs are regular. However, most covering relations are obtained by uncrossing non-adjacent strands, and thus are not $i$-crossing pairs. Nonetheless, we will show that all covering relations are regular pairs. As a consequence, we will show that closed electroid varieties are regular in codimension $1$.

\begin{lem}\label{good i}
Suppose $\tau\lessdot\tau'$ and $\tau$ is obtained from $\tau'$ by uncrossing non-adjacent strands $b$ and $c$. Let $a<b<c<d$ be non-adjacent numbers in cyclic order such that $\tau'(b)=d, \tau'(c)=a$ and $\tau(b)=a, \tau(c)=d$. Suppose there exists $i$ such that $\tau(i)\not=i+1$ and strands $i,i+1$ cross in $\tau'$. We can construct $\sigma$ and $\sigma'$ such that $\sigma\lessdot \tau$ and $\sigma'\lessdot \tau'$ are $i$-crossing pairs, respectively. If either
\begin{enumerate}[label=(\roman*)]
    \item $a\leq i\leq b-1$, or
    \item $d<i\leq a-1, a<\tau(i)<b,$ and $c\leq \tau(i+1)<d$
\end{enumerate}
then $ \sigma\lessdot\sigma'$.
\end{lem}

\begin{proof}
\textbf{Case 1:} If $a<i<b-1$, then strand $i$ and $i+1$ are disjoint from the vertices $a,b,c,d$. So they still cross in $\tau$ and $\sigma < \sigma'$ witnessed by uncrossing strand $(a,c)$ and $(b,d)$ in $\sigma'$ to form strands $(a,b)$ and $(c,d)$ in $\sigma$. It must be a covering relation because we have a graded poset.

If $i=a$, strand $a+1$ crosses with strand $(a,c)$ in $\tau'$ implies that $\tau'(a+1)>c$ in cyclic order. Thus, strand $a+1$ must also cross with the strand $(a,b)$ in $\tau$. Furthermore, $\sigma<\sigma'$ is witnessed by uncrossing strand $(i+1,c)$ and strand $(b,d)$ in $\sigma'$ to form strand $(i+1, b)$ and $(c,d)$ in $\sigma$. And it must be a covering relation because we have a graded poset. This case is illustrated in Figure \ref{f:comm diagram}.

The $i=b-1$ case is similar.

\begin{figure}
    \centering
    \begin{tikzpicture}
        \coordinate (A) at (-3,0);
        \coordinate (B) at (3,0);
        \coordinate (C) at (-3,-3.5);
        \coordinate (D) at (3,-3.5);
        
        \draw (A) circle (1);
        \draw (B) circle (1);
        \draw (C) circle (1);
        \draw (D) circle (1);

        \draw [shift={(-3,0)},thick] (.707,.707) to[out=270,in=90] (.707,-.707);
        \draw [shift={(-3,0)},thick] (-.707,.707) to[out=270,in=90] (-.707,-.707);
        \draw [shift={(-3,0)},thick] (.866,.5) to[out=210,in=10] (-1,0);

         \draw (A) ++(.866,.5) node [right,scale=.8] {$i+1$};
         \draw (A) ++(-1,0) node [left,scale=.8] {$\tau(i+1)$};
         \draw (A) ++(.707,.707) node [above,scale=.8] {$a$};
         \draw (A) ++(.707,-.707) node [below,scale=.8] {$b$};
         \draw (A) ++(-.707,-.707) node [below,scale=.8] {$c$};
         \draw (A) ++(-.707,.707) node [above,scale=.8] {$d$};

         \draw [shift={(3,0)},thick] (.707,.707) to[out=225,in=45] (-.707,-.707);
        \draw [shift={(3,0)},thick] (-.707,.707) to[out=-45,in=135] (.707,-.707);
        \draw [shift={(3,0)},thick] (.866,.5) to[out=210,in=10] (-1,0);

         \draw (B) ++(.866,.5) node [right,scale=.8] {$i+1$};
         \draw (B) ++(-1,0) node [left,scale=.8] {$\tau(i+1)$};
         \draw (B) ++(.707,.707) node [above,scale=.8] {$a$};
         \draw (B) ++(.707,-.707) node [below,scale=.8] {$b$};
         \draw (B) ++(-.707,-.707) node [below,scale=.8] {$c$};
         \draw (B) ++(-.707,.707) node [above,scale=.8] {$d$};

         \draw [shift={(-3,-3.5)},thick] (.707,.707) to[out=225,in=0] (-1,0);
        \draw [shift={(-3,-3.5)},thick] (-.707,.707) to[out=270,in=90] (-.707,-.707);
        \draw [shift={(-3,-3.5)},thick] (.866,.5) to[out=235,in=90] (.707,-.707);

         \draw (C) ++(.866,.5) node [right,scale=.8] {$i+1$};
         \draw (C) ++(-1,0) node [left,scale=.8] {$\tau(i+1)$};
         \draw (C) ++(.707,.707) node [above,scale=.8] {$a$};
         \draw (C) ++(.707,-.707) node [below,scale=.8] {$b$};
         \draw (C) ++(-.707,-.707) node [below,scale=.8] {$c$};
         \draw (C) ++(-.707,.707) node [above,scale=.8] {$d$};

         \draw [shift={(3,-3.5)},thick] (.707,.707) to[out=225,in=0] (-1,0);
        \draw [shift={(3,-3.5)},thick] (-.707,.707) to[out=-45,in=135] (.707,-.707);
        \draw [shift={(3,-3.5)},thick] (.866,.5) to[out=210,in=45] (-.707,-.707);

         \draw (D) ++(.866,.5) node [right,scale=.8] {$i+1$};
         \draw (D) ++(-1,0) node [left,scale=.8] {$\tau(i+1)$};
         \draw (D) ++(.707,.707) node [above,scale=.8] {$a$};
         \draw (D) ++(.707,-.707) node [below,scale=.8] {$b$};
         \draw (D) ++(-.707,-.707) node [below,scale=.8] {$c$};
         \draw (D) ++(-.707,.707) node [above,scale=.8] {$d$};

         \draw (-4.5,1) node {\fbox{$\tau$}};
         \draw (1.5,1) node {\fbox{$\tau'$}};
         \draw (-4.5,-2.5) node {\fbox{$\sigma$}};
         \draw (1.5,-2.5) node {\fbox{$\sigma'$}};

    \end{tikzpicture}
    \caption{In Case 1 with $a=i$, $\sigma <\sigma'$ is witnessed by uncrossing strand $(i+1,c)$ and strand $(b,d)$ in $\sigma'$ to form strands $(i+1, b)$ and $(c,d)$ in $\sigma$.}
    \label{f:comm diagram}
\end{figure}
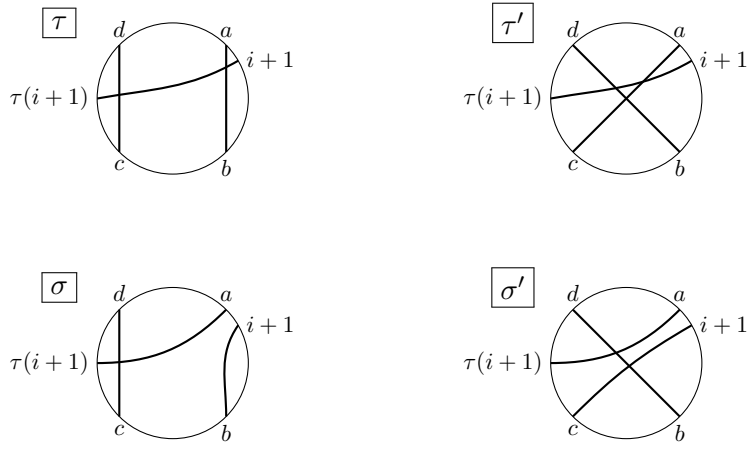

\textbf{Case 2:} Since $\tau(i)<b$, strand $i$ must also cross strand $(a,b)$ in $\tau$. Furthermore, $\sigma<\sigma'$ is witnessed by resolving the crossing between strands $(i,c)$ and $(b,d)$ in $\sigma'$ to form strands $(i,b)$ and $(c,d)$ in $\sigma$. This is a covering relation because we have a graded poset.
\end{proof}

We will show that under the hypotheses of Lemma \ref{good i}, $\sigma\lessdot \sigma'$ is a regular pair if and only if $\tau\lessdot \tau'$ is. This will allow us to extend regularity from $i$-crossing pairs to all pairs. To do this, we construct a morphism``gluing'' the $\psi$ morphisms from Proposition \ref{p:uncrossing adjacent} associated to the two $i$-crossing pairs. Notice that since both are $i$-crossing pairs, the formulae for both $\psi$ are given by $Y\mapsto u_i(a).Y$. Thus, it is relatively easy to define the glued morphism with correct source and target. The following lemma shows that a certain pair of rational functions agree on the intersection of their domains of definition. It will be used to show that inverses of $\psi$ glue.

\begin{lem}\label{pluckers of good i}
Suppose the hypothesis of Lemma \ref{good i} holds. 
\begin{enumerate}[label=(\roman*)]
    \item Under the hypotheses of Lemma \ref{good i}(i), $I_{i+1}(\tau)=I_{i+1}(\tau')$.
    \item Under the hypotheses of Lemma \ref{good i}(ii), $\Delta_{I_{i+1}'(\tau)}=\Delta_{I_{i+1}(\tau)} \cdot \frac{\Delta_{I_{i+1}'}(\tau')}{\Delta_{I_{i+1}}(\tau')}$ as regular functions on $\mathring{\chi}_{\tau'}$ that do not vanish on a dense open set of $\mathring{\chi}_{\tau'}$. Thus, $\frac{\Delta_{I_{i+1}}(\tau')}{\Delta_{I_{i+1}'}(\tau')}=\frac{\Delta_{I_{i+1}}(\tau)}{\Delta_{I_{i+1}'}(\tau)}$ as nonzero rational functions on $\mathring{\chi}_{\tau'}$.
\end{enumerate}

\end{lem}
\begin{proof}

\textbf{Case 1:} We use Lemma \ref{pairing to necklace} to show the equality. The strands in $\tau$ and $\tau'$ are identical except at vertices $a,b,c,d$, where in $\tau$ we have strands $(a,b),(c,d)$ and in $\tau'$ we have strands $(a,c),(b,d)$. If $a+1\leq i+1<b$, strand $i+1$ is disjoint from the vertices $a,b,c,d$. So the cyclically smaller elements among the two strands are $b$ and $c$ for both $\tau$ and $\tau'$. If $i+1=b$, then we only need to consider the contribution from $(c,d)$ in $\tau'$ and $(a,c)$ in $\tau$, which are both $c$. So $I_{i+1}(\tau)=I_{i+1}(\tau').$

\textbf{Case 2:} Fix any $X\in \mathring{\chi}_{\tau'}$ represented by a matrix with columns given by $v_1,\cdots,v_{2n}$. Let $f'$ denote the electrical permutation for $\tau'$. Strand $i$ and $i+1$ cross in $\tau'$, so $f(i)<f(i+1)$. By Lemma \ref{recover a}, $\lambda=\frac{\Delta_{I_{i+1}'(\tau')}(X)}{\Delta_{I_{i+1}(\tau')}(X)}$ is the unique number such that $v_i-\lambda v_{i+1}\in \spann\{v_{i+2},\cdots, v_{f(i)}\}$.

Now, by Lemma \ref{pairing to necklace}, we have $I_{i+1}(\tau)=I_{i+1}(\tau')\setminus \{b-1\}\bigcup \{c-1\}$. Since $f(i)<f(i+1)$ implies $i+1\in I_{i+1}(\tau')$, we also have $I_{i+2}(\tau')=I_{i+1}(\tau')\setminus\{i+1\}\bigcup \{ f(i+1)\}$. Combining the above, we have \begin{align*}
    I_{i+1}(\tau)\setminus\{i+1\}\bigcup \{j\}&=I_{i+1}(\tau')\setminus\{i+1,b-1\}\bigcup \{j,c-1\}\\&=I_{i+2}(\tau')\setminus\{b-1, f(i+1)\}\bigcup \{j,c-1\}.
\end{align*} 
Now, since $f(i+1)=\tau(i+1)-1\geq c-1\geq i+2$ and for each $i+2\leq j\leq f(i)=\tau(i)-1< b-1$, we have $I_{i+1}(\tau)\setminus\{i+1\}\bigcup \{j\}<_{i+2} I_{i+2}(\tau')$. So $\Delta_{I_{i+1}(\tau)\setminus\{i+1\}\bigcup \{j\}}=0$ on $\mathring{\chi}_{\tau'}$ for any $i+2\leq j\leq f(i)$.

Now, take the linear expression $v_i-\lambda v_{i+1}\in \spann\{v_{i+2},\cdots, v_{f(i)}\}$ and append columns in $I_{i+1}(\tau)\setminus\{i+1\}$ to form minors. Every term on the RHS is zero. So we have $\lambda \Delta_{I_{i+1}(\tau)} = \Delta_{I_{i+1}'(\tau)}$ on $\mathring{\chi}_{\tau'}$. This gives a factorization $\Delta_{I_{i+1}'(\tau)}=\Delta_{I_{i+1}(\tau)} \cdot \frac{\Delta_{I_{i+1}'}(\tau')}{\Delta_{I_{i+1}}(\tau')}$ as regular functions on $\mathring{\chi}_{\tau'}$

Let $V$ be the hypersurface in the Grassmannian given by the vanishing of $\Delta_{I_{i+1}(\tau)}$. If $V$ contains $\mathring{\chi}_{\tau'}$, it must also contain its Zariski closure, which we showed to contain $\Xo_\tau$ in Theorem \ref{Main1}.5. This is a contradiction because $V$ has empty intersection with $\Xo_\tau$. Also, we showed in Theorem \ref{Main1}.1 that $\Xo_{\tau'}$ is irreducible. So $V\bigcap\mathring{\chi}_{\tau'} $ has codimension at least $1$. Thus, $\Delta_{I_{i+1}(\tau)}$ does not vanish on a dense open subset of $\mathring{\chi}_{\tau'} $.

The rational function $\lambda$ does not vanish on a dense open set by Proposition \ref{p:uncrossing adjacent} and Lemma \ref{recover a}. So $\Delta_{I_{i+1}'(\tau)}=\lambda \Delta_{I_{i+1}(\tau)}$ does not vanish on a dense open set and we have an equality of nonzero rational functions $\frac{\Delta_{I_{i+1}}(\tau')}{\Delta_{I_{i+1}'}(\tau')}=\frac{1}{\lambda}=\frac{\Delta_{I_{i+1}}(\tau)}{\Delta_{I_{i+1}'}(\tau)}$ on $\mathring{\chi}_{\tau'}$.
\end{proof}

\begin{lem}\label{equiv of pairs}
Suppose the hypothesis of Lemma \ref{good i} holds. Then, we have an open embedding of algebraic varieties
\begin{align*}
    \phi: (\mathring{\chi}_{\sigma} \bigsqcup \mathring{\chi}_{\sigma'})\times\mathbb{C}^* &\longrightarrow \Xo_\tau\bigsqcup \mathring{\chi}_{\tau'}\\
    (Y,a) &\longmapsto u_i(a).Y
\end{align*}
onto the dense open subscheme of the target given by the union of the nonvanishing of $\Delta_{I_{i+1}'(\tau')}$ and the nonvanishing of $\Delta_{I_{i+1}'(\tau)}$. 

The image contains the generic point of $\Xo_\tau$. And the inverse is given by $X\mapsto (u_i(-a).X,a)$, where $a=\frac{\Delta_{I_{i+1}}(\tau')}{\Delta_{I_{i+1}'}(\tau')}=\frac{\Delta_{I_{i+1}}(\tau)}{\Delta_{I_{i+1}'}(\tau)}$.
\end{lem}
\begin{proof}
The target $\mathring{\chi}_{f}\bigsqcup \mathring{\chi}_{f'}$ has complement in $\chi_{f'}$ given by $\sqcup\Xo_g$ where the disjoint union is over all $g<f'$ with $g\not=f$. We want to show that the complement equals the union of closed electroid varieties ${\chi}_g$ for $g\lessdot f'$ and $g\not=f$. {For any $\Xo_g$ with $g<f'$, there is some chain $g\lessdot f_k\lessdot f_{k-1} \lessdot\dots \lessdot f_1 \lessdot f_0\lessdot f'$. By Proposition \ref{2 coverings}, there exists some $h\not=f$ such that $f_1\lessdot h \lessdot f'$,} and thus $\Xo_g\subset \chi_h$. The other direction of containment is clear. We conclude that the target $\mathring{\chi}_{f}\bigsqcup \mathring{\chi}_{f'}$ is dense open in $\chi_{f'}$ and thus is a variety. Similarly, the source is a dense open subset of $\chi_{\sigma'}\times\CC^*$ and thus is also a variety. 

By Proposition \ref{p:uncrossing adjacent}, we can respectively define $\mathring{\chi}_{\sigma}\times\mathbb{C}^* \longrightarrow \mathring{\chi}_{\tau} $ and $\mathring{\chi}_{\sigma'}\times\mathbb{C}^* \longrightarrow \mathring{\chi}_{\tau'} $ both by $(Y,a) \longmapsto u_i(a).Y$. And the images are given respectively by $\{\Delta_{I_{i+1}'(\tau)}\not=0\}\bigcap \mathring{\chi}_{\tau}$ and $\{\Delta_{I_{i+1}'(\tau')}\not=0\}\bigcap \mathring{\chi}_{\tau'}$. We show that under the hypotheses of Lemma \ref{good i}(i) or \ref{good i}(ii) their union is a dense open set in $\Xo_\tau\bigsqcup \mathring{\chi}_{\tau'}$.

Under the hypothesis of Lemma \ref{good i}(i), Lemma \ref{pluckers of good i}(i) gives ${I_{i+1}'(\tau')}={I_{i+1}'(\tau)}$. So the set theoretic union of the two images is $\{\Delta_{I_{i+1}'(\tau')}\not=0\}\bigcap (\mathring{\chi}_{\tau}\bigsqcup \mathring{\chi}_{\tau'})$, which is open in $\mathring{\chi}_{\tau}\bigsqcup \mathring{\chi}_{\tau'}$.

Under the hypothesis of Lemma \ref{good i}(ii), Lemma \ref{pluckers of good i}(ii) gives $$\Delta_{I_{i+1}'(\tau)}\Delta_{I_{i+1}(\tau')}=\Delta_{I_{i+1}(\tau)} \Delta_{I_{i+1}'(\tau')}$$ on $\mathring{\chi}_{\tau'}$, and thus on its Zariski closure $\chi_{\tau'}$ by Theorem \ref{Main1}.(5).

On $\mathring{\chi}_{\tau}$, $\Delta_{I_{i+1}(\tau)}$ is everywhere nonzero because it is in the Grassmann necklace. So we have a factorization $\Delta_{I_{i+1}'(\tau)}\Delta_{I_{i+1}(\tau')}\frac{1}{\Delta_{I_{i+1}(\tau)}}=  \Delta_{I_{i+1}'(\tau')} $ as regular functions on $\mathring{\chi}_{\tau}$. Thus, the nonvanishing of $\Delta_{I_{i+1}'(\tau)}$ in $\mathring{\chi}_{\tau}$ contains the nonvanishing of $\Delta_{I_{i+1}'(\tau')}$ in $\mathring{\chi}_{\tau}$. We conclude that $$\{\Delta_{I_{i+1}'(\tau)}\not=0\}\bigcap \mathring{\chi}_{\tau}=(\{\Delta_{I_{i+1}'(\tau)}\not=0\}\bigcup \{\Delta_{I_{i+1}'(\tau')}\not=0\})\bigcap \mathring{\chi}_{\tau}.$$

On $\mathring{\chi}_{\tau'}$, $\Delta_{I_{i+1}(\tau')}$ is everywhere nonzero because it is in the Grassmann necklace. So we have a factorization $\Delta_{I_{i+1}'(\tau)}=\Delta_{I_{i+1}(\tau)}  \Delta_{I_{i+1}'(\tau')} \frac{1}{\Delta_{I_{i+1}(\tau')}}$ as regular functions on $\mathring{\chi}_{\tau'}$. Thus, the nonvanishing of $\Delta_{I_{i+1}'(\tau')}$ in $\mathring{\chi}_{\tau'}$ contains the nonvanishing of $\Delta_{I_{i+1}'(\tau)}$ in $\mathring{\chi}_{\tau'}$. We conclude that $$\{\Delta_{I_{i+1}'(\tau')}\not=0\}\bigcap \mathring{\chi}_{\tau'}=(\{\Delta_{I_{i+1}'(\tau)}\not=0\}\bigcup \{\Delta_{I_{i+1}'(\tau')}\not=0\})\bigcap \mathring{\chi}_{\tau'}.$$

By the above two equalities, the set theoretic union of the two images equals $$(\{\Delta_{I_{i+1}'(\tau)}\not=0\}\bigcup \{\Delta_{I_{i+1}'(\tau')}\not=0\})\bigcap (\Xo_\tau \bigsqcup\mathring{\chi}_{\tau'} )$$ and this is indeed open in $\Xo_\tau \bigsqcup\mathring{\chi}_{\tau'}$.

We may a priori define the morphism $\phi:(Y,a) \longmapsto u_i(a).Y
$ from $(\mathring{\chi}_{\sigma} \bigsqcup \mathring{\chi}_{\sigma'})\times\mathbb{C}^* $ to the Grassmannian given by $(Y,a)\mapsto u_i(a).Y$. The set theoretic image of $\phi$ is the union of the set theoretic images of the morphisms $\mathring{\chi}_{\sigma}\times\mathbb{C}^* \longrightarrow \mathring{\chi}_{\tau} $ and $\mathring{\chi}_{\sigma'}\times\mathbb{C}^* \longrightarrow \mathring{\chi}_{\tau'} $, which we showed is an open set in $\Xo_\tau \bigsqcup\mathring{\chi}_{\tau'}$ given by the union of the nonvanishing of $\Delta_{I_{i+1}'(\tau)}$ and $\Delta_{I_{i+1}'(\tau')}$. In particular, this open set contains the image of $\mathring{\chi}_{\sigma'}\times\mathbb{C}^* \longrightarrow \mathring{\chi}_{\tau'} $, which is dense in $\mathring{\chi}_{\tau'}$ by Proposition \ref{p:uncrossing adjacent}, and thus dense in $\chi_{\tau'}$ by Theorem \ref{Main1}.5. So the Zariski closure of the set theoretic image of $\phi$ is the closed positroid variety $\chi_{\tau'}$. The source is reduced, so the scheme theoretic image of $\phi$ is $\chi_{\tau'}$. Thus, $\phi$ factors through the closed embedding of $\chi_{\tau'}$ into the Grassmannian. Since the image lands in the open set $ \Xo_\tau\bigsqcup \mathring{\chi}_{\tau'}$ of $\chi_{\tau'}$, we have constructed the morphism $\phi$.

To show that $\phi$ is an open embedding onto its image, we need to define a rational inverse for $\phi$, i.e., the last claim of the lemma. And this follows from the inverse formula for $\psi$ from Proposition \ref{p:uncrossing adjacent} and the fact that the two rational functions agree with each other in from Lemma \ref{pluckers of good i}.
\end{proof}

\begin{cor}\label{regular pairs}   
Under the hypotheses of Lemma \ref{good i}, $\sigma\lessdot \sigma'$ is a regular pair iff and only if $\tau\lessdot \tau'$ is a regular pair.
\end{cor}

\begin{lem}\label{boundary genericly smooth}
Let $\tau\lessdot\tau'$. Then, the closed electroid variety $\chi_{\tau'}$ is regular along the generic point of the open electroid variety $\Xo_{\tau}$.  
\end{lem}
\begin{proof}
By Proposition \ref{p:uncrossing adjacent strong}, $i$-crossing pairs are regular. By Corollary \ref{regular pairs}, it suffices to show that moves from Lemma \ref{good i} connect any pair $\tau\lessdot \tau'$ to an $i$-crossing pair. This is done by a combinatorial algorithm which explicitly constructs such a path. 

The strategy of the algorithm is to apply a ``regular pair''-preserving move from Lemma \ref{good i} to iteratively decrease the distance between strands $a$ and $d$ until the strands are adjacent. At each iteration, $\tau'(a-1)$ lies in one of three intervals, $(a,b),(c,d),(d,a)$, each dictating a different move. Since $\tau\lessdot\tau'$ is a covering relation, $\tau'(a-1)$ cannot be in the cyclic interval $(b,c)$. 

\begin{itemize}
    \item Case 1: If $c<\tau'(a-1)<d$, the hypotheses of Lemma \ref{good i}(i) are satisfied; cross strands $a$ and $a-1$.
    \item Case 2: If $d<\tau'(a-1)<a$, the hypotheses of Lemma \ref{good i}(i) are satisfied; cross strands $a$ and $a-1$.
    \item Case 3: If $a<\tau'(a-1)<b$, the hypotheses of Lemma \ref{good i}(ii) are satisfied; uncross strands $a$ and $a-1$.
\end{itemize}
\end{proof}

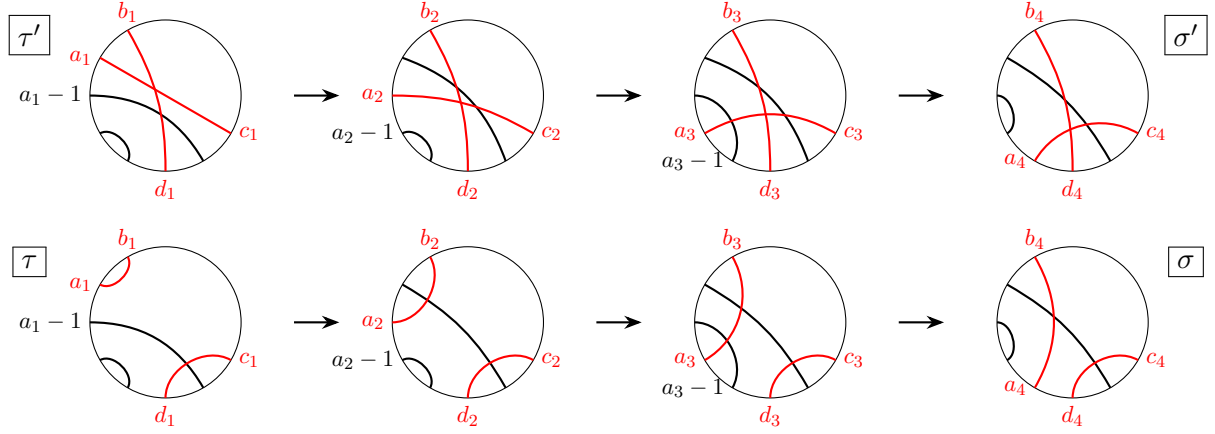
\begin{figure}[h]
\centering
\begin{tikzpicture}

\coordinate (A1) at (-6,0);
\coordinate (A2) at (-2,0);
\coordinate (A3) at (2,0);
\coordinate (A4) at (6,0);

\coordinate (B1) at (-6,-3);
\coordinate (B2) at (-2,-3);
\coordinate (B3) at (2,-3);
\coordinate (B4) at (6,-3);

\draw (A1) circle (1);
\draw (B1) circle (1);
\draw (A2) circle (1);
\draw (B2) circle (1);
\draw (A3) circle (1);
\draw (B3) circle (1);
\draw (A4) circle (1);
\draw (B4) circle (1);

\draw (-7.8,.8) node {\fbox{$\tau'$}};
\draw (-7.8,-2.2) node {\fbox{$\tau$}};

\draw (7.5,.8) node {\fbox{$\sigma'$}};
\draw (7.5,-2.2) node {\fbox{$\sigma$}};

\draw (A1) ++(-.866,.5) node [left, red ,scale=.8] {$a_1$};
\draw (A1) ++(-.5,.866) node [above, red,scale=.8] {$b_1$};
\draw (A1) ++(.866,-.5) node [right, red,scale=.8] {$c_1$};
\draw (A1) ++(0,-1) node [below, red,scale=.8] {$d_1$};
\draw (A1) ++(-1,0) node [left, scale = .8] {$a_1-1$};

\draw (B1) ++(-.866,.5) node [left, red,scale=.8] {$a_1$};
\draw (B1) ++(-.5,.866) node [above, red,scale=.8] {$b_1$};
\draw (B1) ++(.866,-.5) node [right, red,scale=.8] {$c_1$};
\draw (B1) ++(0,-1) node [below, red,scale=.8] {$d_1$};
\draw (B1) ++(-1,0) node [left, scale = .8] {$a_1-1$};

\draw (A2) ++(-1,0) node [left, red, scale = .8] {$a_2$};
\draw (A2) ++(-.5,.866) node [above, red, scale = .8] {$b_2$};
\draw (A2) ++(.866,-.5) node [right, red, scale = .8] {$c_2$};
\draw (A2) ++(0,-1) node [below, red, scale = .8] {$d_2$};
\draw (A2) ++(-.866,-.5) node [left, scale = .8] {$a_2-1$};

\draw (B2) ++(-1,0) node [left, red, scale = .8] {$a_2$};
\draw (B2) ++(-.5,.866) node [above, red, scale = .8] {$b_2$};
\draw (B2) ++(.866,-.5) node [right, red, scale = .8] {$c_2$};
\draw (B2) ++(0,-1) node [below, red, scale = .8] {$d_2$};
\draw (B2) ++(-.866,-.5) node [left, scale = .8] {$a_2-1$};

\draw (A3) ++(-.866,-.5) node [left, red ,scale=.8] {$a_3$};
\draw (A3) ++(-.5,.866) node [above, red,scale=.8] {$b_3$};
\draw (A3) ++(.866,-.5) node [right, red,scale=.8] {$c_3$};
\draw (A3) ++(0,-1) node [below, red,scale=.8] {$d_3$};
\draw (A3) ++(-.5,-.866) node [left, scale = .8] {$a_3-1$};

\draw (B3) ++(-.866,-.5) node [left, red,scale=.8] {$a_3$};
\draw (B3) ++(-.5,.866) node [above, red,scale=.8] {$b_3$};
\draw (B3) ++(.866,-.5) node [right, red,scale=.8] {$c_3$};
\draw (B3) ++(0,-1) node [below, red,scale=.8] {$d_3$};
\draw (B3) ++(-.5,-.866) node [left, scale = .8] {$a_3-1$};

\draw (A4) ++(-.5,-.866) node [left, red, scale = .8] {$a_4$};
\draw (A4) ++(-.5,.866) node [above, red, scale = .8] {$b_4$};
\draw (A4) ++(.866,-.5) node [right, red, scale = .8] {$c_4$};
\draw (A4) ++(0,-1) node [below, red, scale = .8] {$d_4$};

\draw (B4) ++(-.5,-.866) node [left, red, scale = .8] {$a_4$};
\draw (B4) ++(-.5,.866) node [above, red, scale = .8] {$b_4$};
\draw (B4) ++(.866,-.5) node [right, red, scale = .8] {$c_4$};
\draw (B4) ++(0,-1) node [below, red, scale = .8] {$d_4$};

\draw[-Stealth, thick] (-4.3,0) -- (-3.7,0);
\draw[-Stealth, thick] (-4.3,-3) -- (-3.7,-3);
\draw[-Stealth, thick] (-.3,0) -- (.3,0);
\draw[-Stealth, thick] (-.3,-3) -- (.3,-3);
\draw[-Stealth, thick]  (3.7,0) -- (4.3,0);
\draw[-Stealth, thick]  (3.7,-3) -- (4.3,-3);

\draw [shift={(-6,0)},thick] (-.866,-.5) to[out=30,in=60] (-.5,-.866);
\draw [shift={(-6,0)},thick] (-1,0) to[out=0,in=120] (.5,-.866);
\draw [shift={(-6,0)},red,thick] (-.866,.5) to[out=-30,in=150] (.866,-.5);
\draw [shift={(-6,0)},red,thick] (0,-1) to[out=90,in=-60] (-.5,.866);

\draw [shift={(-6,-3)}, thick] (-.866,-.5) to[out=30,in=60] (-.5,-.866);
\draw [shift={(-6,-3)},thick] (-1,0) to[out=0,in=120] (.5,-.866);
\draw [shift={(-6,-3)},red,thick] (-.866,.5) to[out=-30,in=-60] (-.5,.866);
\draw [shift={(-6,-3)},red,thick] (0,-1) to[out=90,in=150] (.866,-.5);

\draw [shift={(-2,0)},thick] (-.866,-.5) to[out=30,in=60] (-.5,-.866);
\draw [shift={(-2,0)},thick] (-.866,.5) to[out=-20,in=110] (.5,-.866);
\draw [shift={(-2,0)},red,thick] (-1,0) to[out=0,in=150] (.866,-.5);
\draw [shift={(-2,0)},red,thick] (0,-1) to[out=90,in=-60] (-.5,.866);

\draw [shift={(-2,-3)}, thick] (-.866,-.5) to[out=30,in=60] (-.5,-.866);
\draw [shift={(-2,-3)},thick] (-.866,.5) to[out=-30,in=120] (.5,-.866);
\draw [shift={(-2,-3)},red,thick] (-1,0) to[out=0,in=-60] (-.5,.866);
\draw [shift={(-2,-3)},red,thick] (0,-1) to[out=90,in=150] (.866,-.5);

\draw [shift={(2,0)},thick] (-1,0) to[out=0,in=60] (-.5,-.866);
\draw [shift={(2,0)},thick] (-.866,.5) to[out=-20,in=110] (.5,-.866);
\draw [shift={(2,0)},red,thick] (-.866,-.5) to[out=30,in=150] (.866,-.5);
\draw [shift={(2,0)},red,thick] (0,-1) to[out=90,in=-60] (-.5,.866);

\draw [shift={(2,-3)}, thick] (-1,0) to[out=0,in=60] (-.5,-.866);
\draw [shift={(2,-3)},thick] (-.866,.5) to[out=-30,in=120] (.5,-.866);
\draw [shift={(2,-3)},red,thick] (-.866,-.5) to[out=30,in=-60] (-.5,.866);
\draw [shift={(2,-3)},red,thick] (0,-1) to[out=90,in=150] (.866,-.5);

\draw [shift={(6,0)},thick] (-1,0) to[out=0,in=30] (-.866,-.5);
\draw [shift={(6,0)},thick] (-.866,.5) to[out=-30,in=120] (.5,-.866);
\draw [shift={(6,0)},red,thick] (-.5,-.866) to[out=60,in=150] (.866,-.5);
\draw [shift={(6,0)},red,thick] (0,-1) to[out=90,in=-60] (-.5,.866);

\draw [shift={(6,-3)}, thick] (-1,0) to[out=0,in=30] (-.866,-.5);
\draw [shift={(6,-3)},thick] (-.866,.5) to[out=-30,in=120] (.5,-.866);
\draw [shift={(6,-3)},red,thick] (-.5,-.866) to[out=60,in=-60] (-.5,.866);
\draw [shift={(6,-3)},red,thick] (0,-1) to[out=90,in=150] (.866,-.5);

\end{tikzpicture}
\caption{Under the algorithm, $\tau\lessdot\tau'$ is transformed into an $i$-crossing pair $\sigma\lessdot\sigma'$.}
\label{f:reg_pairs}
\end{figure}

\begin{eg}
    We use the pair $\tau\lessdot\tau'$ depicted in Figure \ref{f:reg_pairs} to illustrate the algorithm transforming $\tau\lessdot\tau'$ into an $i$-crossing pair. For the first move, we are in case 1 and cross strands $a_1,a_1-1$. For the second move, we are in case 2 and cross strands $a_2,a_2-1$. The final move is an application of case 3 and we uncross strands $a_3,a_3-1$.
    \label{e:reg_pair}
\end{eg}

\begin{proof} [Proof of Theorem \ref{Main1}, R1]
By Theorem \ref{cell decomp}, $\chi_f=\bigcup_{g\geq f} \chi_g^\circ$ as a set. Theorem \ref{Main1}.1 gives that each boundary cell has the expected dimension given by the difference in the number of crossings. Combining the fact that open electroid varieties are smooth together with Lemma \ref{boundary genericly smooth} yields the desired result.
\end{proof}

\section{A Splicing Isomorphism and Dense Tori}\label{section:A Splicing Isomorphism and Dense Tori}

Recall that Lam's grove measurement map embeds the totally positive torus $(\mathbb{R}_{>0})^{c(\tau)}$ into the electroid space $\chi_n$ as the  totally nonnegative part of the open electroid variety $(\mathring{\chi}_\tau)_{\geq0}$ \cite{Lam2018}. It is a priori far from true that the grove measurement map extends to an embedding from the algebraic torus to the open electroid variety. First, a map from an alegbraic torus can be injective on the positive real torus but still does not define an embedding into its scheme-theoretic image. An example is given by the $2$-to-$1$ cover $\mathbb{C}^*\to \mathbb{C}^*$ sending $z\mapsto z^2$. Another potential issue is that the image of the algebraic torus might not still be contained in the same Zariski open set. Indeed, to lie in the open electroid variety requires the nonvanishing of certain Pl\"ucker coordinates, which are polynomials in the edge weight variables. A function on the algebraic torus can certainly have a nontrivial vanishing locus, while maintaining nonzero on the positive part. An example is given by the function $\mathbb{C}^*\to \mathbb{C}$ sending $z\mapsto z^2+1$, where the image of the positive torus $\mathbb{R}_{>0}$ is contained in $\mathbb{C}^*$, but the image of the algebraic torus is not. As a corollary of results in this section, we show that neither of the above happens and the grove measurement map indeed extends to embed the alegbraic torus inside the corresponding open electroid variety.

There is an analogous problem for positroid varieties of whether the boundary measurement map extends to embed the algebraic torus inside an open positroid variety. A positive answer is given by the twist automorphism by Muller and Speyer \cite{Muller_2017}. They provide an automorphism on the open positroid variety, called the twist automorphism, that makes the boundary measurement map an inverse of the face Pl\"ucker map. We take a different approach.

A notable consequence of results in this section is that it lays the foundation for proving the conjecture by Thomas Lam that electroid varieties are \textit{positive geometries}. Positive geometries are defined by Arkani-Hamed, Bai, and Lam \cite{AHBL17} as a triple $(X,X_{\geq0},\Omega_X)$, where $X$ is a projective algebraic variety, $X_{\geq0}$ is a semi-algebraic subset of $X(\mathbb{R})$, and $\Omega_X$ is a top-degree
meromorphic differential form (the canonical form) on $X$. We refer the interested reader to \cite{lam2022invitationpositivegeometries} for more details. Positroid varieties are examples of positive geometries \cite{AHBCGPT16} and their canonical forms are given by pushing-forward the canonical form from the algebraic torus via the boundary measurement map. Our result that the grove measurement map embeds an algebraic torus makes it possible to write out a meromorphic form on the open electroid variety in the exact same way -- we simply take the canonical form on the algebraic torus and push it forward to the open electroid variety. An important feature of the canonical form $\Omega_X$ is that it is required to only have logarithmic poles along the ``boundary" of $X$. In the case of electroid varieties, the boundary of a closed electroid variety $\chi_\tau$ is expected to be the complement of the open electroid variety $\mathring{\chi}_{\tau}$, and thus is expected to be exactly the poles of the canonical form $\Omega_{\chi_\tau}$. In order for this conjecture to hold, the torus embedding must have image contained in $\mathring{\chi}_{\tau}$, because the pushforward form does not have poles along the image of the embedded torus. Our result says that this is indeed the case.

Instead of proving the torus embedding directly, we prove a stronger result that resembles the splicing of open positroids in \cite{gorsky2025,gorsky2024}. Fix $n_1$ and $n_2$. We prove Theorem \ref{Main1}.6 and construct an isomorphism between the product of two electroid spaces, $\chi_{n_1}\times \chi_{n_2}$, and a single higher dimensional electroid variety $\chi_{\mu}\subset\chi_{n_1+n_2}$. As a consequence, we can prove Theorem \ref{Main1}.7 and construct an embedding of an algebraic torus into open electroid varieties that restricts to the grove measurement map on totally nonnegative points. Before stating the theorem, we make precise how to index each relevant combinatorial object. 

For objects associated to $\chi_{n_1}$ and $\chi_{n_1+n_2}$, the indexing follows that used throughout the paper. For objects associated to $\chi_{n_2}$, we have to shift the indexing. We index the Grassmannian $\Gr(n_2-1,2n_2)$ by $\{2n_1+1,2n_1+2, \dots, 2n_1+2n_2\}$; correspondingly, the Pl\"ucker coordinates are indexed by subsets of $\{2n_1+1,2n_1+2, \dots, 2n_1+2n_2\}$. Pairings associated to $\chi_{n_2}$ are thought of as pairings on $\{2n_1+1,2n_1+2, \dots, 2n_1+2n_2\}$. Similarly, the cactus networks have boundary vertices labeled by $v_{n_1+1}, v_{n_1+2}, \dots, v_{n_1+n_2}$, where, as before, vertices $v_i$ sits between points $2i-1$ and $2i$ of the pairing. Lastly, a non-crossing partition $\sigma$ on $\{n_1+1, n_1+2, \dots, n_1+n_2\}$ induces dual non-crossing partitions $(\overline{\sigma},\tilde{\sigma})$ where $\overline{\sigma}$ partitions $\{2n_1+1,2n_1+3, \dots, 2n_1+2n_2-1\}$ and $\tilde{\sigma}$ partitions $\{2n_1,2n_1+2, \dots, 2n_1+2n_2\}$. Notice that the re-indexing of the above objects is consistent with each other in the sense that the relationships between the indexing of the Pl\"ucker coordinates, the cactus network, the pairings, and the noncrossing partitions are the same as earlier in the paper. For example, Lemma \ref{pairing to necklace} still holds under the new indexing.

\begin{defn}
Let $\Gamma_A$ and $\Gamma_B$ be cactus networks in $\chi_{n_1}$ and $\chi_{n_2}$, respectively. Define $\Gamma = \Gamma_A\sqcup \Gamma_B$ to be the cactus network on $n_1+n_2$ boundary vertices formed by taking the disjoint union of $\Gamma_A$ and $\Gamma_B$.
\end{defn}

\begin{figure}[h]
\centering
\begin{tikzpicture}
    \coordinate (A) at (-6,0);
    \coordinate (B) at (-2,0);
    \coordinate (C) at (4,0);
    \coordinate (v1) at (-2.866,.5);
    \coordinate (v2) at (-2+.866,.5);
    \coordinate (v3) at (-2,-1);
    \coordinate (v4) at (-6,-1);
    \coordinate (v5) at (-6,1);

    \coordinate (a1) at (4,1.5);
    \coordinate (a2) at (4.882,1.214);
    \coordinate (a3) at (5.427,.464);
    \coordinate (a4) at (5.427,-.464);
    \coordinate (a5) at (4.882,-1.214);
    \coordinate (a6) at (4,-1.5);
    \coordinate (a7) at (4-.882,-1.214);
    \coordinate (a8) at (4-1.427,-.464);
    \coordinate (a9) at (4-1.427,.464);
    \coordinate (a10) at (4-.882,1.214);

    \coordinate (b1) at (4.464, 1.427);  
    \coordinate (b2) at (5.500, 0.000);   
    \coordinate (b3) at (4.464, -1.427);  
    \coordinate (b4) at (4-1.214, -0.882); 
    \coordinate (b5) at (4-1.214, 0.882); 

    \draw (A) circle (1);
    \draw (B) circle (1);
    \draw (C) circle (1.5);

    \fill (v1) circle (2pt);
    \fill (v2) circle (2pt);
    \fill (v3) circle (2pt);
    \fill (v4) circle (2pt);
    \fill (v5) circle (2pt);

    \fill (b1) circle (2pt);
    \fill (b2) circle (2pt);
    \fill (b3) circle (2pt);
    \fill (b4) circle (2pt);
    \fill (b5) circle (2pt);
    
    \draw [thick] (v1) to[out=310,in=110] node[midway,left] {$b$}(v3);
    \draw [thick] (v2) to[out=230,in=70] node[midway,right] {$a$}(v3);
    \draw [thick] (v4) to[out=90,in=270] node[midway, right] {$c$}(v5);

    \draw (v1) node [above, scale=.8] {$v_1$};
    \draw (v2) node [above, scale=.8] {$v_2$};
    \draw (v3) node [below, scale=.8] {$v_3$};
    \draw (v4) node [below, scale=.8] {$v_4$};
    \draw (v5) node [above, scale=.8] {$v_5$};

    \draw (-7.8,1) node {\fbox{$\Gamma_B$}};
    \draw (-3.8,1) node {\fbox{$\Gamma_A$}};
    \draw (1.3,1) node {\fbox{$\Gamma_A\sqcup \Gamma_B$}};

    \draw (b1) node [above, scale=.8] {$v_1$};
    \draw (b2) node [right, scale=.8] {$v_2$};
    \draw (b3) node [below, scale=.8] {$v_3$};
    \draw (b4) node [left, scale=.8] {$v_4$};
    \draw (b5) node [left, scale=.8] {$v_5$};

    \draw [thick] (b1) to[out=290,in=75] node[midway,left=-.5mm] {$b$}(b3);
    \draw [thick] (b3) to[out=60,in=205] node[midway,right=-.5mm] {$a$} (b2);    
    \draw [thick] (b4) to[out=95,in=265]node[midway,right=-.5mm] {$c$} (b5);

\end{tikzpicture}
\caption{The electrical networks $\Gamma_A$ and $\Gamma_B$, and their disjoint union $\Gamma_A\sqcup\Gamma_B$.}
\label{f:splicing networks}
\end{figure}
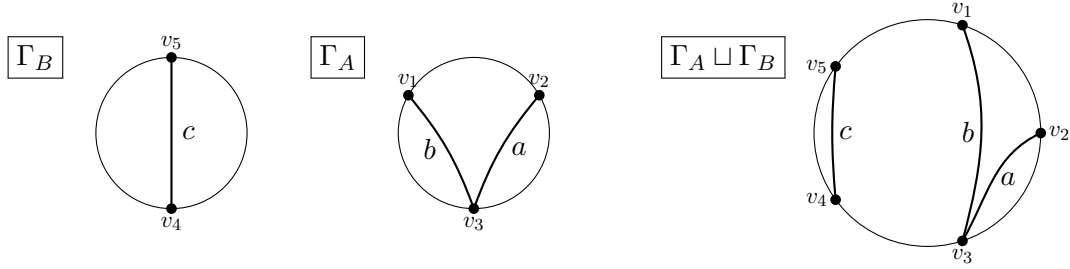

\begin{defn} Let $\tau_1,\tau_2$ be pairings on $[2n_1]$ and $\{2n_1+1,2n_1+2\dots, 2n_1+2n_2\}$, respectively. Define $\tau=\tau_1\sqcup \tau_2$ to be the pairing on $[2n_1+2n_2]$ formed by combining $\tau_1$ and $\tau_2$. That is, $\tau(i):=\tau_1(i)$ for $i\leq 2n_1$ and $\tau(i):=\tau_2(i)$ for $i> 2n_1$.
\end{defn}
\begin{defn} Let $\sigma_1,\sigma_2$ be non-crossing partitions on $[2n_1]$ and $\{2n_1+1,2n_1+2\dots, 2n_1+2n_2\}$, respectively. Define $\sigma=\sigma_1\sqcup \sigma_2$ to be the non-crossing partition on $[2n_1+2n_2]$ formed by the disjoint union of parts of $\sigma_1$ and $\sigma_2$.
\end{defn}

Observe that if $\tau_1,\tau_2$ are medial pairings associated to $\Gamma_A, \Gamma_B$, then $\tau_1\sqcup \tau_2$ is the medial pairing associated to $\Gamma_A\sqcup \Gamma_B$. If a grove in $\Gamma_A$ and a grove in $\Gamma_B$ respectively give non-crossing partitions $\sigma_1$ and $\sigma_2$, then the grove in $\Gamma_A\sqcup \Gamma_B$ formed by their disjoint union give non-crossing partition $\sigma_1\sqcup \sigma_2$. 

Let $\mu_1$ and $\mu_2$  be the maximal pairings on $[2n_1]$ and $\{2n_1+1,2n_1+2\dots, 2n_1+2n_2\}$, respectively. That is, $\chi_{\mu_1}=\chi_{n_1}$ and $\chi_{\mu_2}=\chi_{n_2}$. Let $\mu=\mu_1\sqcup \mu_2$. Observe that the lower order ideal of $\mu$ is isomorphic as a poset to the product of the lower order ideal of $\mu_1$ and the lower order ideal of $\mu_2$. This is because covering relations are given by resolving crossings, and we can resolve crossings of $\mu$ separately in each part. Namely, the lower order ideal of $\mu$ consist of exactly the elements of the form $\tau_1\sqcup \tau_2$ for $\tau_1\leq\mu_1, \tau_2\leq\mu_2$. {In fact, the isomorphism of posets can be extended to an isomorphism of stratifications of closed electroid varieties. Theorem \ref{Main1}.6 can be stated precisely as follows.}

\begin{thm}\label{splicing}
Using the embedding above, there is an isomorphism $\Phi: \chi_{n_1}\times\chi_{n_2}\xrightarrow{\sim} \chi_{\mu}$ between a product of electroid spaces and a closed electroid variety.   
\begin{enumerate}
    \item The map $\Phi$ is given by the following. For any $I\in \binom{[2n_1+2n_2]}{n_1+n_2-1}$, \begin{equation*}
  \Delta_I(\Phi(A,B)) = 
   \begin{cases}
    \Delta_{I_1\setminus\{2n_1\}}(A)\Delta_{I_2}(B) \hspace{1in}& |I_1|={n_1}, 2n_1\in I_1\\
    \Delta_{I_1}(A)\Delta_{I_2\setminus\{2n_1+2n_2\}}(B) & |I_2|={n_2}, 2n_1+2n_2\in I_2\\
    0 &\text{else.}
\end{cases} 
\end{equation*}
where the decomposition $I=I_1\sqcup I_2$ is given by
\begin{align*}
I_1&=I\cap \{1,2,\dots, 2n_1\},\\
I_2&=I\cap \{2n_1+1,2n_1+2, \dots, 2n_1+2n_2\}.
\end{align*}

\item The inverse isomorphism $\Psi$ of $\Phi$ is uniquely determined by the following. Write $\Psi(C)=(A,B)$. Then, under the same embedding to Grassmannians as in $(1)$, we have that
$$\frac{\Delta_{I}(A)}{\Delta_{J}(A)}= \frac{\Delta_{{I}\cup\{2n_1+2n_2\}\cup K}(C)}{\Delta_{{J}\cup\{2n_1+2n_2\}\cup K}(C)}$$
for any subsets $I,J$ of size $n_1-1$ in $\{1, 2,\dots, 2n_1\}$ and any subset $K$ of size $n_2-1$ in $\{2n_1+1, 2n_1+2, \dots, 2n_1+2n_2\}$. Similarly,  we have that
$$\frac{\Delta_{I}(B)}{\Delta_{J}(B)}= \frac{\Delta_{I\cup\{2n_1\}\cup K}(C)}{\Delta_{J\cup\{2n_1\}\cup K}(C)}$$
for any subsets $I,J$ of size $n_2-1$ in $\{2n_1+1, 2n_1+2, \dots, 2n_1+2n_2\}$, and any subset $K$ of size $n_1-1$ in $\{1, 2,\dots , 2n_1\}$.

\item The above isomorphisms $\Phi$ and $\Psi$, restricted to totally nonnegative points, map each pair of cactus networks $(\Gamma_A, \Gamma_B)$ up to electrical equivalence to their disjoint union $\Gamma_A\sqcup\Gamma_B$ up to electrical equivalence.

\item For any pairings $\tau_1,\tau_2$ on $[2n_1]$ and $\{2n_1+1,2n_1+2, \dots, 2n_1+2n_2\}$ respectively, $\Phi$ and $\Psi$ restrict to isomorphisms $\chi_{\tau_1}\times \chi_{\tau_2} \simeq \chi_{\tau_1\sqcup \tau_2}$ and $\Xo_{\tau_1}\times \Xo_{\tau_2} \simeq \Xo_{\tau_1\sqcup \tau_2}$.

\end{enumerate}

\label{t:splicing}
\end{thm}

\begin{proof}
 
\textbf{The map $\Phi$ on nonnegative points.} Recall from \cite{Lam2018} that totally nonnegative points in $\chi_n$ correspond to cactus networks modulo electrical equivalence. It is easy to see that the correspondence defined in $(3)$ on cactus networks is constant on electrical equivalence classes. Thus, $(3)$ defines a correspondence on totally nonnegative points $(\chi_{n_1})_{\geq0}\times(\chi_{n_2})_{\geq0}\simeq (\chi_{\mu})_{\geq0}$.

Let $(\Gamma_A, \Gamma_B)$ be a pair of cactus networks  with $n_1$ and $n_2$ boundary vertices, respectively. Let $\Gamma=\Gamma_A\sqcup \Gamma_B$ be the cactus network on $n_1+n_2$ boundary vertices as in (3). Let $A,B,C$ be the points in $\chi_{n_1},\chi_{n_2},\chi_{\mu}$ represented by $\Gamma_A,\Gamma_B,\Gamma$, respectively. We first show that formulae in $(1)$ and $(2)$ are satisfied under the correspondence $(A,B)\leftrightarrow C$.

Consider $\Gamma$. For $L_\sigma(\Gamma)$ to be nonzero, $\sigma$ must decompose into $\sigma = \sigma_1\sqcup \sigma_2$, where $\sigma_1$ is a partition of $\{1,2,\dots, n_1\}$ and $\sigma_2$ is a partition of $\{n_1+1, n_1+2, \dots, n_1+n_2\}$. Moreover, $2n_1$ and $2n_1+2n_2$ must be in the same block of $\tilde{\sigma}$. Thus, for $J\in\binom{[2n_1+2n_2]}{n_1+n_2-1}$ to be concordant with such $\sigma$, $J$ must contain $2n_1$ or $2n_1+2n_2$. So $\Delta_J(\Gamma)=\sum_{\sigma || J} L_{\sigma(\Gamma)}$ is nonzero only if $J$ contains $2n_1$ or $2n_1+2n_2$.

Fix $J\in\binom{[2n_1+2n_2]}{n_1+n_2-1}$ containing $2n_1$ or $2n_1+2n_2$. Let $J=J_1\sqcup J_2$, where \begin{align*}
J_1&=J\cap \{1,2, \dots,2n_1\},\\   
J_2&=J\cap \{2n_1+1,2n_1+2,\dots, 2n_1+2n_2\}.
\end{align*} 
Note that $\sigma=\sigma_1\bigsqcup \sigma_2$ being concordant with $J$ is equivalent to the existence of $J',J''$ such that
\begin{itemize}
    \item $J'$ is concordant to $\sigma_1$ and $J''$ is concordant to $\sigma_2$,
    \item $J_1\setminus \{2n_1\}\subset J'\subset J_1$ and $ J_2\setminus \{2n_1+2n_2\}\subset J''\subset J_2$.
\end{itemize}
In particular, the first condition implies that $|J'|=n_1-1$ and $|J''|=n_2-1$. Thus, since $|J|=n_1+n_2-1=|J_1|+|J_2|-1$, exactly one of the following holds
\begin{itemize}
    \item $J'=J_1\setminus \{2n_1\}$ and $J''=J_2$.
    \item $J'=J_1$ and $ J''=J_2\setminus \{2n_1+2n_2\}$.
\end{itemize}
The cardinality condition uniquely determines $J'$ and $J''$ for given $J_1\sqcup J_2$. {See Figure \ref{f:J,J',J''} for an example of equivalent concordance conditions.} 

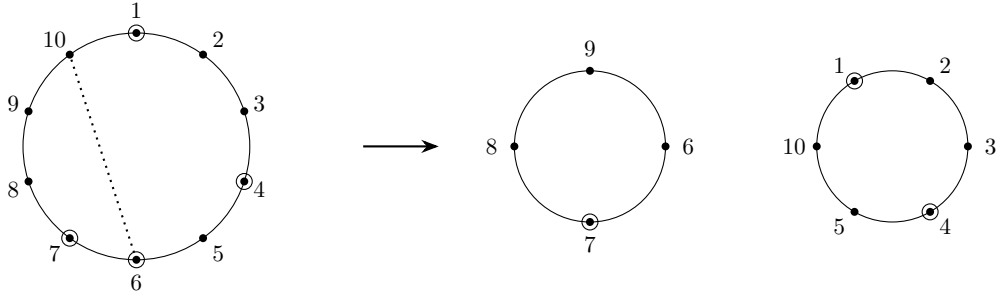
\begin{figure}[h]
\centering
\begin{tikzpicture}
    \coordinate (A) at (6,0);
    \coordinate (B) at (2,0);
    \coordinate (C) at (-4,0);

    \coordinate (a1) at (-4,1.5);
    \coordinate (a2) at (-4+.882,1.214);
    \coordinate (a3) at (-4+1.427,.464);
    \coordinate (a4) at (-4+1.427,-.464);
    \coordinate (a5) at (-4+.882,-1.214);
    \coordinate (a6) at (-4,-1.5);
    \coordinate (a7) at (-4-.882,-1.214);
    \coordinate (a8) at (-4-1.427,-.464);
    \coordinate (a9) at (-4-1.427,.464);
    \coordinate (a10) at (-4-.882,1.214);

    \draw (A) circle (1);
    \draw (B) circle (1);
    \draw (C) circle (1.5);
    \draw [thick, dotted] (a6) -- (a10);

    \draw (A) ++(-1,0) node [xshift=-.3cm,scale=.8] {10};
    \draw (A) ++(-.5,.866) node [xshift=-.2cm,yshift=.2cm,scale=.8] {1};
    \draw (A) ++(.5,.866) node [xshift=.2cm,yshift=.2cm,scale=.8] {2};
    \draw (A) ++(1,0) node [xshift=.3cm,scale=.8] {3};
    \draw (A) ++(.5,-.866) node [xshift=.2cm,yshift=-.2cm,scale=.8] {4};
    \draw (A) ++(-.5,-.866) node [xshift=-.2cm,yshift=-.2cm,scale=.8] {5};

    \fill (A) ++(-1,0) circle (1.5pt);
    \fill (A) ++(-.5,.866) circle (1.5pt);
    \fill (A) ++(.5,.866) circle (1.5pt);
    \fill (A) ++(1,0) circle (1.5pt);
    \fill (A) ++(.5,-.866) circle (1.5pt);
    \fill (A) ++(-.5,-.866) circle (1.5pt);

    \draw (B) ++(0,-1) node [yshift = -.3cm,scale=.8] {7};
    \draw (B) ++(-1,0) node [xshift = -.3cm,scale=.8] {8};
    \draw (B) ++(0,1) node [yshift = .3cm,scale=.8] {9};
    \draw (B) ++(1,0) node [xshift = .3cm,scale=.8] {6};

    \fill (B) ++(0,-1) circle (1.5pt);
    \fill (B) ++(0,1) circle (1.5pt);
    \fill (B) ++(-1,0) circle (1.5pt);
    \fill (B) ++(1,0) circle (1.5pt);

    \draw[-Stealth, thick] (-1,0) -- (0,0);

    \fill (a1) circle (1.5pt) node[yshift =.3cm, scale=.8]{1}; 
    \fill (a2) circle (1.5pt) node[xshift = .2cm, yshift =.2cm, scale=.8]{2}; 
    \fill (a3) circle (1.5pt) node[xshift =.2cm, yshift=.1cm, scale=.8]{3}; 
    \fill (a4) circle (1.5pt) node[xshift =.2cm, yshift=-.1cm, scale=.8]{4}; 
    \fill (a5) circle (1.5pt) node[xshift = .2cm, yshift =-.2cm, scale=.8]{5}; 
    \fill (a6) circle (1.5pt) node[yshift =-.3cm, scale=.8]{6}; 
    \fill (a7) circle (1.5pt) node[xshift=-.2cm,yshift =-.2cm, scale=.8]{7}; 
    \fill (a8) circle (1.5pt) node[xshift =-.2cm, yshift=-.1cm, scale=.8]{8}; 
    \fill (a9) circle (1.5pt) node[xshift =-.2cm, yshift=.1cm, scale=.8]{9}; 
    \fill (a10) circle (1.5pt) node[xshift = -.2cm, yshift =.2cm, scale=.8]{10}; 

    \draw (a1) circle (3pt);
    \draw (a4) circle (3pt);
    \draw (a6) circle (3pt);
    \draw (a7) circle (3pt);
    \draw (B) ++(0,-1) circle (3pt);
    \draw (A) ++(-.5,.866) circle (3pt);
    \draw (A) ++(.5,-.866) circle (3pt);
    
\end{tikzpicture}
\caption{On $\chi_\mu$, calculating whether $\sigma=\sigma_1\sqcup\sigma_2$ is concordant with $\{1,4,6,7\}$ reduces to calculating whether $\sigma_1\parallel \{1,4\}$ and $\sigma_2\parallel \{7\}$.}
\label{f:J,J',J''}
\end{figure}

We can now expand $\Delta_J(C)=\Delta_J(\Gamma)$ as follows.
\begin{align*}
\Delta_J(C) &= \sum\limits_{\sigma\parallel J} L_{\sigma}(\Gamma)\\ &=\sum\limits_{\substack{\sigma_1\parallel J'\\ \sigma_2\parallel J''}} L_{\sigma_1\sqcup \sigma_2}(\Gamma)\\
&= \sum\limits_{\substack{\sigma_1\parallel J'\\ \sigma_2\parallel J''}} L_{\sigma_1}(\Gamma_A)L_{\sigma_2}(\Gamma_B)\\
&=\Big(\sum\limits_{\sigma_1\parallel J'} L_{\sigma_1}(\Gamma_A)\Big)\Big(\sum_{\sigma_2\parallel J''} L_{\sigma_2}(\Gamma_B)\Big)\\
&=\Delta_{J'}(\Gamma_A)\Delta_{J''}(\Gamma_B)\\
&= \Delta_{J'}(A)\Delta_{J''}(B).
\end{align*}
Given this, it is then straightforward to verify the formula in $(2)$ holds as well. Thus, the formulae in $(1)$ and $(2)$ hold under the correspondence $(A,B)\leftrightarrow C$.
So far, we showed that if $\Phi$ and $\Psi$ are well-defined morphisms, then their restriction to totally nonnegative points is given by the correspondence $(A,B)\leftrightarrow C$ given by disjoint union of networks.

\textbf{Well-definedness of the morphism $\Phi$.} The formula is degree-homogeneous. Fix any point $(A,B)\in\chi_{n_1}\times \chi_{n_2}$. By Theorem \ref{Main1}.5, there exists pairings $\tau_1,\tau_2$ such that $A\in \mathring{\chi}_{\tau_1}, B\in \mathring{\chi}_{\tau_2}$. Consider the Grassmann necklaces of $\tau_1,\tau_2$; by Remark \ref{r: no i,i-1}, we have $2n_1\not\in I_{1}(\tau_1)$. Thus, $\Delta_{I_{1}(\tau_1)\sqcup I_{1}(\tau_2)\sqcup \{2n_1\}}(\Phi(A,B))=\Delta_{I_{1}(\tau_1)}(A)\Delta_{I_{1}(\tau_2)}(B)\not=0$ and the formula gives a well-defined morphism $\Phi$ from $\chi_{n_1}\times \chi_{n_2}$ to the projective space. Since the source is reduced and irreducible by Theorem \ref{Main1}.1, the scheme-theoretic image is the Zariski closure of the set-theoretic image, is irreducible, and has dimension at most $\dim \chi_{n_1}+\dim \chi_{n_2}$. On the other hand, the image contains $\Phi((\chi_{n_1})_{\geq0}\times (\chi_{n_2})_{\geq0})$, which we just showed to be $(\chi_{\mu})_{\geq0}$. So the scheme theoretic image contains the Zariski closure of $(\chi_{\mu})_{\geq0}$. By Theorem \ref{Main1}.2 and Theorem \ref{Main1}.5, the Zariski closure of $(\chi_{\mu})_{\geq0}$ equals $\chi_\mu$. By Theorem \ref{Main1}.1 and Theorem \ref{Main1}.5, dimensions of closed electroid varieties can be computed by the number of crossings in the strand diagram. And we see that $\dim \chi_{n_1}+\dim \chi_{n_2}=\dim \chi_\mu$ since $\mu$ is a disjoint union of maximally crossed strand diagrams $\mu_1$ and $\mu_2$. Thus, the scheme theoretic image of $\Phi$ is $\chi_\mu$.

\textbf{Well-definedness of the inverse morphism $\Psi$.}
Write $\Psi(C)=(A,B)$. It suffices to prove well-definedness for the first component $\chi_\mu\to \chi_{n_1}$ of $\Psi$, and the argument for the second component follows similarly. First, we check that the ratios in the definition of $\Psi$ do not depend on the choice of $K$. This is an equality of rational functions on $\chi_{\mu}$; as $(\chi_{\mu})_{\geq0}$ is Zariski dense in $\chi_{\mu}$, it suffices to check equality on cactus networks.

Fix $C\in (\mathring{\chi}_{\tau})_{\geq0}$ for $\tau\leq \mu$. By the discussion above the theorem, $\tau$ decomposes as a disjoint union $\tau_1\sqcup \tau_2$. Then, $C$ can be represented by a cactus network $\Gamma$ which decomposes into disjoint cactus networks $\Gamma_A$ and $\Gamma_B$ as in $(3)$. By definition, each grove coordinate $L_\sigma$ equals $\sum \wt(F)$ where the sum is over all groves $F$ with boundary partition $\sigma$. Since $\Gamma = \Gamma_A\sqcup\Gamma_B$, the grove $F$ splits into $F=F_A\sqcup F_B$ with $F_A\subset\Gamma_A, F_B\subset\Gamma_B$. Moreover, since all nonzero grove coordinates $L_\sigma$ on $\chi_\mu$ have that $2n_1$ and $2n_1+2n_2$ are in the same block of $\tilde{\sigma}$, we can expand each $L_\sigma$ as $\sum \wt(F_A)\wt(F_B)$ where the sum is over all groves $F_A$ on $\{1,2,\dots,{2n_1}\}$ and all groves $F_B$ on $\{{2n_1}+1, {2n_1}+2,\dots, 2n_1+{2n_2}\}$ with appropriate concordance. Thus,
\begin{align*}
  \frac{\Delta_{I_1}(A)}{\Delta_{J_1}(A)}:=\frac{\Delta_{I_1\cup\{2n_1+2n_2\}\cup K}(C)}{\Delta_{J_1\cup\{2n_1+2n_2\}\cup K}(C)} &=\frac{\sum\limits_{\sigma\parallel I_1\cup \{2n_1+2n_2\}\cup K}L_\sigma}{\sum\limits_{\mu\parallel J_1\cup \{2n_1+2n_2\}\cup K}L_\mu}\\
  &= \frac{\sum\limits_{\substack{F_A\subset\Gamma_A : \sigma(\Gamma_A) \parallel I_1\\ F_B\subset\Gamma_B : \sigma(\Gamma_A) \parallel K}}\wt(F_A)\wt(F_B)}{\sum\limits_{\substack{F_A\subset\Gamma_A : \sigma(\Gamma_A) \parallel J_1\\F_B\subset\Gamma_B : \sigma(\Gamma_A) \parallel K}}\wt(F_A)\wt(F_B)}\\
  &= \frac{\bigg(\sum\limits_{F_A\subset\Gamma_A : \sigma(\Gamma_A) \parallel I_1}\wt(F_A)\bigg)\bigg(\sum\limits_{F_B\subset\Gamma_B : \sigma(\Gamma_A) \parallel K}\wt(F_B)\bigg)}{\bigg(\sum\limits_{F_A\subset\Gamma_A : \sigma(\Gamma_A) \parallel J_1}\wt(F_A)\bigg)\bigg(\sum\limits_{F_B\subset\Gamma_B : \sigma(\Gamma_A) \parallel K}\wt(F_B)\bigg)}\\
  &= \frac{\sum\limits_{F_A\subset\Gamma_A : \sigma(\Gamma_A) \parallel I_1}\wt(F_A)}{\sum\limits_{F_A\subset\Gamma_A : \sigma(\Gamma_A) \parallel J_1}\wt(F_A)}.
\end{align*}
This is clearly independent of the choice of $K$. Also, as an immediate consequence, the cocycle identity $\frac{\Delta_I(A)}{\Delta_J(A)}\cdot \frac{\Delta_J(A)}{\Delta_K(A)}\cdot  \frac{\Delta_I(A)}{\Delta_K(A)}=1$ holds for the image of $\Psi$ under the defining formula in $(2)$. Thus, the ratios are well-defined rational functions on $\chi_\mu$ and they collectively define a rational morphism $\Psi$ from $\chi_\mu$ to the projective space. This also directly shows that $\Psi$, restricted to nonnegative points, maps $\Gamma_A\sqcup \Gamma_B$ to $(\Gamma_A, \Gamma_B)$.

Now, we show that the rational morphism to projective space is, in fact, regular. Since $\chi_\mu$ is an reduced and irreducible variety by Theorem \ref{Main1}.3, rational morphisms have a maximal domain of definition given by the union of domain of representatives. Notice that on the nonvanishing of $\Delta_{{J}\cup\{2n_1+2n_2\}\cup K}$, the morphism $\Psi$ is well-defined by setting $\Delta_{J}:=1$ and $\Delta_{I}:=\Delta_{J}\cdot\frac{\Delta_{{I}\cup\{2n_1+2n_2\}\cup K}(C)}{\Delta_{{J}\cup\{2n_1+2n_2\}\cup K}(C)} $ for all $I$. So it suffices to show that the nonvanishing of $\Delta_{{J}\cup\{2n_1+2n_2\}\cup K}$'s cover $\chi_\mu$.

Fix $C\in\chi_\mu$. Then, $C\in \mathring{\chi}_{\tau}$ for some $\tau\leq \mu$. By the discussion above the theorem, $\tau=\tau_1\sqcup \tau_2$. By Lemma \ref{pairing to necklace}, the Grassmann necklace of $\tau$ has that ${I_{2n_1+1}(\tau)}=J_0\cup\{2n_1+2n_2\}\cup K_0$, where $J_0$ is a subset of size $n_1-1$ in $\{1, 2,\dots, 2n_1-1\}$ and $K_0$ is a subset of size $ n_2-1$ in $\{2n_1+1, 2n_1+2, \dots, 2n_1+2n_2-1\}$. Thus, we have identified a nonvanishing Pl\"ucker for each fixed $C\in \mathring{\chi}_{\tau}$. 

So far, we showed that $\Psi$ is a well defined morphism from $\chi_\mu$ to the product of projective spaces. To show that the scheme theoretic image is $\chi_{n_1}\times \chi_{n_2}$, we can use the same argument as for $\Phi$. The source is reduced and irreducible by Theorem \ref{Main1}.1. We already checked that the restriction of $\Psi$ to totally nonnegative points of $\chi_{\mu}$ maps to $(\chi_{n_1})_{\geq0} \times (\chi_{n_2})_{\geq0}$, which is Zariski dense in $\chi_{n_1}\times \chi_{n_2}$ by Theorem \ref{Main1}.2 and Theorem \ref{Main1}.5. Finally, dimensions match by Theorem \ref{Main1}.1 and Theorem \ref{Main1}.5. This shows that $\Psi$ defines a morphism $\chi_{\mu}\to \chi_{n_1}\times \chi_{n_2}$.

\textbf{Morphisms $\Phi$ and $\Psi$ are two sided inverses.} The fact that $\Psi\circ\Phi=Id$ is a basic alegbraic manipulation. Notice that $\Phi\circ\Psi$ restricted to totally nonnegative points of $\chi_\mu$ is the identity. Reduced and separated shows that $\Phi\circ\Psi$ must be identity on the Zariski closure of $(\chi_\mu)_{\geq0}$, which is $\chi_\mu$.

\textbf{Morphisms $\Phi$ and $\Psi$ respect the electroid stratification.} From (3), it is clear that on totally nonnegative points $\Phi((\chi_{\tau_1})_{\geq 0}\times(\chi_{\tau_2})_{\geq 0}) = (\chi_{\tau_1\sqcup \tau_2})_{\geq 0}$ as schemes. Consider the restriction $\Phi|_{\chi_{\tau_1}\times\chi_{\tau_2}}$. Since the source is reduced and irreducible, the scheme-theoretic image is the Zariski closure of the set-theoretic image, is irreducible, and has dimension at most $c(\tau_1)+c(\tau_2)$. Thus, the image contains the Zariski closure of $(\chi_{\tau_1\sqcup \tau_2})_{\geq 0}$ which is $\chi_{\tau_1\sqcup \tau_2}$. As the dimension of $\chi_{\tau_1\sqcup \tau_2}$ is $c(\tau_1)+c(\tau_2)$, the image of $\Phi|_{\chi_{\tau_1}\times\chi_{\tau_2}}$ is $\chi_{\tau_1\sqcup \tau_2}$. Thus, $\Phi$ restricts to an isomorphism $\chi_{\tau_1}\times \chi_{\tau_2} \simeq \chi_{\tau_1\sqcup \tau_2}$. 

Recall that electroid varieties form a stratification. Thus, the image $\Phi(\Xo_{\tau_1}\times \Xo_{\tau_2})$ equals the complement of $\bigcup \Phi(\chi_{\sigma_1}\times \chi_{\sigma_2})$, where the union is taken over all pairs $(\sigma_1,\sigma_2)\not=(\tau_1,\tau_2)$ such that $\sigma_1\leq \tau_1,\sigma_2\leq \tau_2$. We just showed that $\Phi(\chi_{\sigma_1}\times \chi_{\sigma_2})=\chi_{\sigma_1\sqcup \sigma_2}$. Thus, the image of $\Xo_{\tau_1}\times \Xo_{\tau_2} $ is the complement of the union of all such $\chi_{\sigma_1\sqcup \sigma_2}$. Since the product of two stratifications gives a stratification, the image equals $ \Xo_{\tau_1\sqcup \tau_2}$. Thus, $\Phi$ restricts to an isomorphism $\Xo_{\tau_1}\times \Xo_{\tau_2} \simeq \Xo_{\tau_1\sqcup \tau_2}$ as well.
\end{proof}

\begin{cor}\label{isomorphism of lower electroid space}
Let $\tau$ be a pairing on $[2n]$ with some strand $(i,i+1)$ and all other strands maximally crossed. Then, there exists an isomorphism between the closed electroid variety $\chi_\tau\subset \chi_n$ and $\chi_{n-1}$, respecting the stratification on both sides by electroid
varieties. Furthermore, restricted to totally nonnegative points, this isomorphism is given by removing an isolated boundary vertex or merging two adjacent shorted boundary vertices, depending on the parity of $i$.
\label{strandi,i+1}
\end{cor}

\begin{proof}
By cyclic symmetry as in Subsection \ref{cyclic symmetry}, we can assume that $i=1$. Then, it is a direcct consequence of Theorem \ref{t:splicing} where $n_1=1$.
\end{proof}
\begin{proof}[Proof of Theorem \ref{Main1}.7]
We follow the same procedure as in \cite[Section 5.10]{Lam2018}, which was stated only for totally nonnegative points. However, their approach generalizes once we have Proposition \ref{p:uncrossing adjacent} defined for open electroid varieties and Corollary \ref{isomorphism of lower electroid space} defined for closed electroid varieties.

Every pairing $\tau$ either contains strands $i$ and $i+1$ which cross or contains a strand from $i$ to $i+1$. To see this, let $i$ and $j$ be strands that cross with $i<j$ and $j-i$ minimal. If $j\not=i+1$, then all strands with one endpoint between $i$ and $j$ have both endpoints between $i$ and $j$. By the minimality of $j-i$, none of these strands cross each other and so $i+1$ is connected to $j-1$, $i+2$ is connected to $j-2$, etc. Thus, there is a strand from $(i+j-1)/2$ to $(i+j+1)/2$.

From this observation, the proof follows immediately. If strands $i$ and $i+1$ cross, let $\tau'$ be pairing such that $\tau'\lessdot \tau$ is an $i$-crossing pair. Proposition \ref{p:uncrossing adjacent} gives that $\mathring{\chi}_\tau\simeq \mathring{\chi}_{\tau'}\times\CC^*$. Otherwise, there's a strand from $i$ to $i+1$; let $\tau''$ be the pairing on $[2n-2]$ formed by removing strand $(i,i+1)$ from $\tau$. Corollary \ref{strandi,i+1} gives that $\mathring{\chi}_\tau\simeq \mathring{\chi}_{\tau''}$. Furthermore, both isomorphisms, restricted to totally nonnegative points, are given by contracting or deleting an edge, and removing an isolated vertex or merging two adjacent shorted vertices, respectively. The result holds by induction on the dimension. 
\end{proof}

The isomorphism $\Phi$ has a simple expression using matrix representatives.

\begin{prop}The isomorphism $\Phi: \chi_{n_1}\times\chi_{n_2}\xrightarrow{\sim} \chi_{\mu}$ stated in Theorem \ref{splicing} can be represented by a map of matrices.
For $(A,B)\in\chi_{\mu_1}\times \chi_{\mu_2}\subset \Gr(n_1-1,2n_1)\times \Gr(n_2-1,2n_2)$, select matrix representatives of $A$ and $B$. Then the point $\Phi(A,B)\in \Gr(n_1+n_2-1,2n_1+2n_2)$ is represented by the $(n_1+n_2-1)\times (2n_1+2n_2)$ matrix, $$ C = \left[\begin{array}{cccccccc|cccc}
0 && &\dots& & & 0 & (-1)^{n_1+1} & 0 & \dots & 0 & (-1)^{n_1+n_2}\\
\hline
 &  &  &  &  & & && && & \\
  &  &  &  &  & A & & & && 0&\\
   &  &  &  &  & & && & \\
   \hline
    &  &  &  &  &  & & & && &\\
        &  &  &  &  & 0& & & && B & \\
\end{array}\right].$$

\end{prop}

\begin{proof}

Observe that the row span of $C$ only depends on the row span of $A$ and $B$, but not the matrices $A$ and $B$. Thus, the matrix map defines a rational map to $\Gr(n_1+n_2-1,2n_1+2n_2)$ whenever $C$ is full rank. To show that this rational map is regular and is equal to $\Phi$, it suffices to show that $\Delta_I(C) = \Delta_I(\Phi(A,B))$. We prove by doing case work according to the definition of $\Phi$.

As before, the decomposition $I=I_1\sqcup I_2$ is given by
\begin{align*}
I_1&=I\cap \{1,2,\dots, 2n_1\},\\
I_2&=I\cap \{2n_1+1,2n_1+2, \dots, 2n_1+2n_2\}.
\end{align*}

Given $J\subset[2n_1+2n_2]$, let $C_J$ denote the submatrix of $C$ given by taking the columns indexed by $J$. And similarly, one can define $A_{J_1}$ and $B_{J_2}$.

Case 1a: Suppose that $|I_1|=n_1$ and $2n_1\in I_1$. If $2n_1+2n_2\not\in I_2$, then by expanding along the top row we get 
\begin{align*}
    \Delta_I(C) &= (-1)^{n_1+1}(-1)^{{n_1+1}}\Delta_{I_1\setminus\{2n_1\}}(A)\Delta_{I_2}(B)\\
    &=\Delta_{I_1\setminus\{2n_1\}}(A)\Delta_{I_2}(B).
\end{align*}

If $n_1+n_2\in I_2$, then by expanding along the top row we get 
$$\Delta_I(C)= \Delta_{I_1\setminus\{2n_1\}}(A)\Delta_{I_2}(B) + (-1)^{n_1+n_2}(-1)^{n_1+n_2}\Delta_{I_1}(A)\Delta_{I_2\setminus\{n_1+n_2\}}(B),$$
which, since the submatrix $A_{I_1}$ is not full rank, becomes
$$\Delta_I(C) = \Delta_{I_1\setminus\{2n_1\}}(A)\Delta_{I_2}(B).$$

Case 1b: Suppose that $|I_2|=n_2$ and $2n_1+2n_2\in I_2$. Then, the proof proceeds as in Case 1a.

Case 2ab: Suppose that $|I_1|=n_1$ but $2n_1\not\in I_1$. Then, the top row of the submatrix $C_{I_1}$ is the zero vector. Thus, the submatrix $C_{I_1}$ fails to be full rank and so $C_I$ fails too. Hence,  $\Delta_I(C) = 0$. 

Case 2b: Suppose $|I_2|=n_2$ but $2n_1+2n_2\not\in I_2$, then, as in Case 2a, $\Delta_I(C)=0.$ 

Case 3: Suppose $|I_1|\not=n_1$ and$|I_2|\not=n_2$. Then, either $|I_1|> n_1$ or $|I_2|> n_2$; without loss of generality, assume $|I_1|> n_1$. Since the rank of $C_{I_1\setminus\{2n_1\}}$ is at most the rank of $A$, the rank of $C_{I_1}$ is at most $rk(A)+1 = n_1$. Thus, $C_{I_1}$ is not full rank and so $\Delta_I(C)=0$. 
\end{proof}

\begin{eg}
To illustrate how the map $\Phi$ works on Pl\"ucker coordinates, matrices, and cactus networks, we revisit the networks in Figure \ref{f:splicing networks}. Let $A$ be the point represented by $\Gamma_A$ and $B$ be the point represented by $\Gamma_B$. We now calculate $\Delta_{1467}(\Phi(A,B))$ using our various representations.

By the map on Pl\"ucker coordinates, 
\begin{align*}
\Delta_{1467}(\Phi(A,B)) &= \Delta_{14}(A)\Delta_7(B)\\
&=(L_{12|3}(\Gamma_A)+L_{13|2}(\Gamma_A))(L_{45}(\Gamma_B))\\
&=(0+b)(c)=bc.
\end{align*}

This aligns with the map on cactus networks, as 
\begin{align*}
\Delta_{1467}(\Gamma) &= \sum\limits_{\substack{\sigma_1\sqcup\sigma_2\parallel 1467\\\sigma_1 \vdash\{1,2,3\}, \sigma_2\vdash\{4,5\}}}L_\sigma(\Gamma)\\
&=L_{13|2|45}(\Gamma)+L_{12|3|45}(\Gamma)=bc.
\end{align*}

Lastly, $A$ and $B$ can be represented by the matrices, $$ A = \left[\begin{array}{cccccc}
b& 1 & 0 & 0 & -b & -1\\
-b&0&a&1&a+b&1
\end{array}\right]$$

$$B = \left[\begin{array}{cccc}
c & 0 & c & 0 \\
\end{array}\right]$$

with image, 

$$ \Phi(A,B) = \left[\begin{array}{cccccc|cccc}
0&0&0&0&0&1&0&0&0&-1\\
\hline
b& 1 & 0 & 0 & -b & -1&0&0&0&0\\
-b&0&a&1&a+b&1&0&0&0&0\\
\hline
0&0&0&0&0&0&c&0&c&0\\
\end{array}\right].$$

Calculating the minor indexed by columns $1,4,6,7$ yields $\Delta_{1467}(\Phi(A,B)) = bc$.

\label{e:3maps}
\end{eg}

\section{Open Problems}
\subsection{Cohen-Macaulayness, normality, and Frobenius Splitting}

We showed that closed electroid varieties are regular in codimension one. We also showed that they are compatibly split under the same Frobenius splitting, and thus they are semi-normal \cite[Proposition~1.2.5]{BrionKumar}. 

\begin{conj}
    Closed electroid varieties are Cohen-Macaulay. Thus, they are normal and are the only compatibly Frobenius split subvarieties in the electroid space $\chi$ under our splitting.
\end{conj}
Cohen-Macaulayness implies normality by Serre's criterion. Normality implies that they are all the compatibly split subvarieties by \cite[Theorem 5.3]{KLS11} and the fact that they form a stratification and the open electroids are regular.
    
\subsection{Positive Geometry}    
We formally state the conjecture by Thomas Lam mentioned in the beginning of Section \ref{section:A Splicing Isomorphism and Dense Tori}.
\begin{conj}
The electroid space $\chi_n$ is a normal positive geometry. The iterative boundary components are exactly the closed electroid varieties and boundary relations are exactly covering relations in the poset of electroids. Furthermore, on any closed electroid variety, its canonical form is given by pushing-forward the canonical form on the dense torus, via any grove measurement embedding given by Theorem \ref{Main1}.7. 
\end{conj}

\bibliographystyle{amsplain} 
\bibliography{arxiv}

\end{document}